\documentclass[12pt]{article}
\title{Joint tridiagonalization}
\author{}
\usepackage{amsfonts,mathrsfs, stmaryrd}
\usepackage{graphicx}
\usepackage{amsmath}
\usepackage{amsthm}
\usepackage{amssymb}
\usepackage{xparse}
\usepackage{color}
\usepackage[colorlinks]{hyperref}

    \newtheorem{theorem}{Theorem}
    \newtheorem{lemma}[theorem]{Lemma}
    \newtheorem{proposition}[theorem]{Proposition}
    \newtheorem{corollary}[theorem]{Corollary}

\theoremstyle{definition} 
    
    \newtheorem{remark}[theorem]{Remark}

\newcommand{\eps}{\varepsilon}

\newcommand{\R}{{\mathbb R}}
\newcommand{\C}{{\mathbb C}}

\newcommand{\N}{{\mathbb N}}

\newcommand{\lstar}{{\raise-0.15ex\hbox{$\scriptstyle \ast$}}}

\theoremstyle{remark} 

\newcommand{\Bp}{\underline{p}}

\newcommand{\rr}{\mathbb{R}}

\newcommand{\Exp}{\mathbb{E}}

\newcommand{\tr}{\operatorname{tr}}
\newcommand{\filt}{\mathscr{F}}
\DeclareDocumentCommand \one { o }
{%
\IfNoValueTF {#1}
{\mathbf{1}  }
{\mathbf{1}\left\{ {#1} \right\} }%
}

\newcommand{\lawequals}{\overset{\mathscr{L}}{=}}
\newcommand{\Prto}{\overset{\Pr}{\to}}

\newcommand{\Span}{\operatorname{Span}}

\newcommand{\UM}{\mathfrak{U}}

\newcommand{\SCS}{s^{\mathfrak{s}}}
\newcommand{\SQ}{s^{Q}}

\newcommand{\TS}{\mathbb{T}}

\DeclareDocumentCommand \M { m d() o }
{
        \IfNoValueTF {#3}
        {
        	\IfNoValueTF {#2}
            	{ {#1} }
		{ {#1}_{(#2)} }
        }
        {
            {#1}_{#3}
        }
}
\DeclareDocumentCommand \MM { m d() d() }
{
        	\IfNoValueTF {#2}
            	{ {#1} }
		{ {#1}_{(#2)}^{(#3)} }
}
\DeclareDocumentCommand \MP { m g d() o }
{
    \IfNoValueTF {#2}
    {
        \IfNoValueTF {#4}
        {
        	\IfNoValueTF {#3}
		{{\boldsymbol #1}}
		{ {\boldsymbol #1}_{(#3)} }
        }
        {
            {\boldsymbol #1}_{#4}
        }
    }
    {
        \IfNoValueTF {#4}
        {
        	\IfNoValueTF {#3}
            	{ {#1}_{#2} }
		{ {#1}_{(#3),#2} }
        }
        {
            {#1}_{#4,#2}
        }
    }
}

\begin{document}

\title{Tridiagonal models for Dyson Brownian motion}
\author{Diane Holcomb and Elliot Paquette}

\maketitle

\begin{abstract}
In this paper, we consider tridiagonal matrices the eigenvalues of which evolve according to $\beta$-Dyson Brownian motion.  This is the stochastic gradient flow on $\R^n$ given by, for all $1 \leq i \leq n,$
\[
  d\lambda_{i,t} = \sqrt{\frac{2}{\beta}}dZ_{i,t} - \biggl( \frac{V'(\lambda_i)}{2} - \sum_{j: j \neq i} \frac{1}{\lambda_i - \lambda_j} \biggr)\,dt
\]
where $V$ is a constraining potential and $\left\{ Z_{i,t} \right\}_1^n$ are independent standard Brownian motions.  This flow is stationary with respect to the distribution
\[
  \rho^{\beta}_N(\lambda) = \frac{1}{Z^{\beta}_N} e^{-\frac{\beta}{2}
  \left( -\sum_{1 \leq i \neq j \leq N} \log|\lambda_i - \lambda_j| + \sum_{i=1}^N V(\lambda_i) \right)
}.
\]
The particular choice of $V(t)=2t^2$ leads to an eigenvalue distribution constrained to lie roughly in $(-\sqrt{n},\sqrt{n}).$  We study evolution of the entries of one choice of tridiagonal flow for this $V$ in the $n\to \infty$ limit. 

On the way describing the evolution of the tridiagonal matrices we give the derivative of the Lanczos tridiagonalization algorithm under perturbation.

\end{abstract}

\tableofcontents

\section{Introduction}

Dyson introduced his model of $N$ Brownian particles evolving in a constraining potential in 1962 \cite{Dyson62}. The particle evolution described by Dyson with quadratic potential gives the evolution of the eigenvalues a Hermitian matrix Brownian motion $M_t$. For fixed time $t>0$ the matrix $M_t$ has Gaussian Unitary Ensemble (GUE) distribution. Dyson also formulated a more general particle flow which corresponds to a larger class of point processes called $\beta$-ensembles. The $\beta$-ensembles are finite point processes with $N$ points and joint density
\begin{equation}
\label{eq:betaensemble}
  \rho^{\beta}_N(\lambda) = \frac{1}{Z^{\beta}_N} e^{-\frac{\beta}{2}
  \left( -\sum_{1 \leq i \neq j \leq N} \log|\lambda_i - \lambda_j| + \sum_{i=1}^N V(\lambda_i) \right)
},
\end{equation}
with $Z_N^{\beta}$ a scaling factor $V(x)$ the potential, and $\beta>0$. The case $V(x)=x^2/2$ and $\beta=1,2, $ and $4$ correspond to the original Gaussian ensembles and their Dyson evolution. These could be further generalized by changing the interaction term, but we restrict ourselves to this case. This joint density leads to the generalization of the Dyson Brownian motion. Let 
\begin{equation}
\label{eq:dlambda}
  d\lambda_{i,t} = \sqrt{\frac{2}{\beta}}dZ_{i,t} - \biggl( \frac{V'(\lambda_i)}{2} - \sum_{j: j \neq i} \frac{1}{\lambda_i - \lambda_j} \biggr)\,dt
\end{equation}
where $\left\{ Z_{i,t} \right\}_1^n$ are independent standard Brownian motions. Then (\ref{eq:betaensemble}) is the equilibrium measure for this process, and the flow is stationary with respect to this measure. 

There is significant existing work on Dyson's model. Rogers and Shi showed convergence to the Wigner law in the $N\to \infty$ limit for the original quadratic potential \cite{RogersShi}. Israelsson showed the fluctuations converged to a Gaussian process with later work by Bender giving the covariance structure \cite{Israelsson, Bender}. Other work on the fluctuations of the trace moments was done by Perez-Abreu and Tudor \cite{AbreuTudor}. Later work by Unterberger extended this to the general $V,\beta$ case \cite{Unterberger}. This work done in the asymptotic setting rests heavily on convergence of the Stieltjes transform but gives little information on local interactions.

In the case of the Gaussian ensembles where $V(x)=x^2/2$ and $\beta=1, 2,$ or $4$ the eigenvalues of the matrices form a finite determinantal or Pfaffian point process. The Dyson flow for $\beta=2$ can also be described as a determinantal point process with the extended Hermite kernel $K(\overline x,\overline y)$, now with points in space-time. These descriptions give exact formulas for correlation functions and may be used to prove local limits. Forrester and Nagao give local point process limits at multiple times via determinantal methods for the $\beta=2$ case \cite{FN2}. Under the appropriate substitutions the extended Hermite kernel converges to the extended sine kernel. Similar results hold at the edge of the spectrum. 

Work in recent years has been focused on the more general Dyson model introduced in (\ref{eq:dlambda}), but much remains to be done. Landon, Sosoe, and Yau give universality results relating eigenvalue statistics for a class of potentials $V$ to the classical Dyson Brownian motion \cite{LSY}. Huang and Landon give results on rigidity and a mesoscopic CLT \cite{HuangLandon}. In particular the question of local limits for the general $\beta$ remain open even in the case of the quadratic potential. These questions are the motivation behind the current work. To see the connection we give a brief review of the work done for non-dynamic model.

In the case of the $\beta$-Hermite ensembles (not time-evolving) for general $\beta>0$ the proof of local limits rests on a tridiagonal matrix models introduced by Dumitriu and Edelman \cite{DE}. The tridiagonal (Jacobi) model for the case $V(x) = 2 x^2$ is given by 
\begin{equation}
\label{eq:tridiagonal}
A_{n} = \left[ \begin{array}{ccccc}
b_1 & a_1 & && \\
a_1 & b_2 & a_2 & &\\
& \ddots & \ddots & \ddots \\
& & a_{n-2} & b_{n-1} & a_{n-1}\\
&&& a_{n-1} & b_n
\end{array}\right]
\end{equation}
with $a_i \sim \chi_{\beta(n-k)/4}$ and $b_i \sim \mathcal{N}(0,2)$ all independent. Notice that the potential in this case is off by a factor of 4 from the potential mentioned previously. This will be a convenient choice later. Dumitriu and Edelman showed that the eigenvalues of $A_n$ have joint density given by (\ref{eq:betaensemble}) \cite{DE}. They in fact go further, using a bijection between the entries of a Jacobi matrix and the eigenvalues together with the square of their spectral weights. 
\begin{proposition}[Dumitriu-Edelman, \cite{DE}]
We take $q_i$ to be the first entry of the vector $v_i$ that satisfies $A_n v_i = \lambda_i v_i$ with $\|v_i \|= 1$, and $\lambda_1< \lambda_2<\cdots < \lambda_n$ to be the ordered eigenvalues of $A_n$. Then the map $(\overline a, \overline b) \leftrightarrow (\overline \lambda, \overline q^2)$ is a bijection, and moreover for $A_n$ defined as above we get 
\begin{equation}
\label{eq:dirichlet}
(q_1^2,q_2^2,...,q_n^2) \sim \operatorname{Dirichlet}(\tfrac{\beta}{2},...,\tfrac{\beta}{2}).
\end{equation}
\end{proposition}
This tridiagonal model was used by Valk\'o and Vir\'ag to give a point process limit in the bulk \cite{BVBV}, and Ram\'irez, Rider and Vir\'ag to give a point process limit at the edge \cite{RRV}.

\subsection{Notation}

We deal with matrix-valued random processes.  The subscripts of such processes are used both for referring to matrix and vector entries, as well as a time parameter.  So, we adopt the following convention.  A process $\MP{X} = \left( \MP{X}{t}, t\geq 0\right)$ will be shown in bold whenever its time parameter is not shown.  Otherwise, it will be shown in normal weight font.  Matrix indexing subscripts will also appear in the subscript, but always first.  So for example, $\MP{X}{t}[i,j]$ refers to the $i,j$ entry of the matrix process $\MP{X}$ at time $t,$ while $\MP{X}[i,j]$ refers to the scalar-valued process of the $i,j$ entry.  

There are many instances in this paper where $[A,B]=AB-BA$ denotes the commutator. In order to avoid confusion all quadratic variation and covariation terms will be written in the form $dX_t dY_t$ or $(dX_t)^2$.

\subsection{Goals and Results}

The goal of this paper is to give a description of the tridiagonal matrix evolution associated to the $\beta$-Dyson Brownian motion flow. These models can in fact be completely characterized, but require a substantial amount of machinery to describe. In this paper we begin by showing the existence of such a tridiagonal flow and then developing the linear algebra tools needed to characterize the model. We then move on to characterizing the model and asymptotics for the case $V(x)= 2x^2$. 

Recall that the tridiagonal model is in bijection with the eigenvalues and the spectral weights. We have an explicit description of the eigenvalue flow, but we still have the freedom to choose the process by which the spectral weights evolve. The first partial description of the tridiagonal model is in Theorem \ref{thm:martingales} which gives a description of the martingale terms for the general model. We then specialize to the case where the spectral weights remain fixed. The model here is given in its entirety by Theorem \ref{thm:tridiagonalfrozen}. This model is particularly nice in the case where the eigenvalue flow is stationary. In this case we get that the tridiagonal model is also stationary. We describe two more cases, the first is the case where the spectral weights evolve according to some finite variation process. The description is given in Theorem \ref{thm:tridiagonalFV}. The second is a model where the martingale terms in the evolution vanish on the first half of the matrix.

In addition to explicit descriptions of the evolution of the finite matrices we get the following asymptotic result on the evolution of the entries of $\MP{A}$ as $n \to \infty$.  Define the operator $\mathcal{F}$ on Jacobi matrices by
  \[
    \mathcal{F}(A)_{k,\ell}
    =
    \begin{cases}
      -b_\ell(4\ell-2)+
       4\sum_{j=1}^{\ell-1} b_j & \text{ if } k = \ell , \\
      -2 (\ell+1)a_\ell- 2\ell a_{\ell-1} 
      + 4\sum_{j=1}^{\ell-2} a_j 
      & \text{ if } k = \ell + 1 ,\\
    0, & \text{ otherwise.}
    \end{cases}
  \]
  and $D$ be the infinite Jacobi matrix
\[
  D
  =
  \begin{bmatrix}
    0 & \frac{1}{2} &   &  &  & \\
    \frac{1}{2} & 0 & \frac{1}{2} &  &  & \\
      & \frac{1}{2} & 0 & \frac{1}{2}&  & \\
      &   & \frac{1}{2} &  &  &  & \\
      &   &   &  &  \ddots &  & \\
  \end{bmatrix}.
\]

Let $\mathfrak{s}$ denote the semicircle density on $[-1,1],$
\[
  \mathfrak{s}(x) = \frac{2}{\pi}\sqrt{1-x^2}\,\one[|x| \leq 1],
\]
and define for $z \in \C \setminus [-1,1]$
\[
  \SCS(z) =
  \int_{-1}^1 \frac{\mathfrak{s}(x)\,dx}{x-z}
  =
  2(-z+\sqrt{z^2-1})
  .
\]

 \begin{theorem}
 \label{thm:asymptotics}
 Let $\MP{A}$ satisfy the evolution for the case of frozen spectral weights in Theorem \ref{thm:tridiagonalfrozen}, with $V(x)=2x^2.$  Suppose that at time $0,$ the eigenvalue distribution of $\MP{A}{0}^{(n)}/\sqrt{n}$ satisfies
 \[
   n^{-1/2}\tr\left( z - \MP{A}{0}^{(n)}/\sqrt{n} \right)^{-1}
   - n^{1/2} \SCS(z)
   \Prto \int_{-1}^1 \frac{\nu(dx)}{x-z}
 \]
 for some signed measure $\nu$
 uniformly on compact sets of $\mathbb{C}\setminus [-1,1].$  Suppose that the spectral weights of $\MP{A}$ satisfy \eqref{eq:dirichlet} and are independent of the eigenvalues of $\MP{A}$.  If $\MP{A}{0}$ is the tridiagonal in (\ref{eq:tridiagonal}), then this is satisfied with $\nu = 0.$ 
 
Then any bounded order principal submatrix of $\MP{A}-D\sqrt{n}$ converges to the solution of
\[
dA_t = \mathcal{F}(A_t) - { \sqrt{\beta}} \mathcal{G}
\]
where $\mathcal{G}$ is a fixed (non-time dependent) Gaussian vector with an explicitly computable correlation structure. If we choose $\overline q$ to be flat (that is $q_i^2 = \frac{1}{n}$ for all $i$) then we get
\[
dA_t = \mathcal{F}(A_t).
\]
\end{theorem}
In the stationary case, or more generally when $\nu$ is $0,$ the limiting evolution of the tridiagonal matrix is identically $0.$

\subsubsection{Derivatives of the Lanczos algorithm}

Finding an explicit description for the tridiagonal evolution requires the us to solve for the tridiagonalization of a symmetric or Hermitian matrix after a small perturbation. This may be used to give an explicit derivative of the Lanczos algorithm. 

We will use $\mathscr{T}$ to denote the tridiagonalization operator on Hermitian matrices. Let $M$ be any $n \times n$ Hermitian matrix so that $\mathscr{T}(M)=A$ where the off diagonal entries are strictly positive. We will take $p_0(x),...,p_{k-1}(x)$ to the be orthogonal polynomials that are orthonormal with respect to the weights $\{q_i\}_{i=1}^n$. See Section \ref{sec:OPs} for further information about the construction. We further take $p_{k}^{(\ell)}(x)$ to denote the $\ell$-th minor polynomials. We use a notation that may not be consistent with other notations choices for the minor polynomials, see Section \ref{sec:minors} for details. 

We make one further definition. We use $D_{G, \mathscr{T}}$ to denote the derivative of the Lanczos algorithm in the direction of $G$ in the tridiagonal basis. In particular 
\[
D_{G,\mathscr{T}} \mathscr{T}(M) = \lim_{\eps \to 0} \frac{1}{\eps} \mathscr{T}(A+\eps G).
\]
Notice that on the right hand side we are taking the tridiagonalization of a perturbed tridiagonal matrix rather than in the original basis for $M$.

\begin{theorem}
\label{thm:lanczos}
Let $M$ be any $n \times n$ Hermitian matrix so that $\mathscr{T}(M)=A$ where the off diagonal entries are strictly positive, and let $W^{k,\ell}$ be the matrix that is 0 everywhere except $(k,\ell)$ and $(\ell,k)$ where it is 1. Let $G$ be a matrix and $D_{G,\mathscr{T}}$ denote the derivative in the direction of $G$ in the tridiagonal basis. Then 
\[
D_{W^{k,\ell},\mathscr{T}} \mathscr{T}(M) = G^{k,\ell}(A)
\]
where 
  \[
    G^{k,\ell}_{u,r}
    =
    \frac{1}{2}\sum_{i=1}^n
    q_i^2
    p_{k-1}(\lambda_i)
    \begin{cases}
    a_{r-1} p_{r-2}(\lambda_i)p_{r-1}^{(\ell-1)}(\lambda_i) - a_r p_{r-1}(\lambda_i) p_{r}^{(\ell-1)}(\lambda_i)&  \\
    +a_{r-1} p_{r-1}(\lambda_i)p_{r-2}^{(\ell-1)}(\lambda_i) - a_r p_{r}(\lambda_i) p_{r-1}^{(\ell-1)}(\lambda_i),  & \text{ if } u = r  < n, \\
    a_r (p_{r-1}(\lambda_i)p_{r-1}^{(\ell-1)}(\lambda_i) - p_r(\lambda_i)p_r^{(\ell-1)}(\lambda_i)), & \text{ if } u = r + 1 < n,\\
    0, & \text{ otherwise.}
  \end{cases}
\]
\end{theorem}

\begin{remark}
The previous theorem is enough to give the derivative for an arbitrary $G$. We can also consider the derivative in the original basis for $M$, this will be equivalent to $D_{O^{*}GO}$ where $O$ is the change of basis matrix $O^*MO=A$.
\end{remark}

\subsection{Organization}

The paper is organized in the following way: section \ref{sec:iso} characterizes matrix diffusions that have the same spectral process, and in section \ref{sec:dysoniso} specializes these matrix diffusions to the case of the $\beta$-Dyson Brownian motion flow. These sections will essentially give existence of the tridiagonal model, but the description of the evolution itself will be implicit. Section \ref{sec:commutator} will develop the linear algebra and orthogonal polynomial tools that will be needed to write down an explicit representation of the tridiagonal matrix evolution. In this section we also prove Theorem \ref{thm:lanczos}. The proof is not explicitly identified, but this theorem will be a consequence of remark \ref{rem:Gkl}. 

Section \ref{sec:commutator} gives a brief introduction to discrete orthogonal polynomials (see section \ref{sec:OPs}) and later defines the corresponding minor polynomials (see section \ref{sec:minors}). The remainder of the sections is split into the derivation of several identities related to the commutator of $[A,M]$ for a tridiagonal $A$, and ends with a summary of the identities derived in the section. This final summary section \ref{sec:summary} gives a summary of the relations needed to describe the tridiagonal model explicitly and the necessary pieces for Theorem \ref{thm:lanczos}.

This explicit representation will be written down in several cases in section \ref{sec:models}. The organization of this sections is pretty self-explanatory from the section heading, but it is worth noting that section on the tridiagonal model with frozen weights \ref{sec:tridiagonalfrozen} will be necessary for reading the section on finite variation spectral weights \ref{sec:tridiagonalfv}. Additionally section \ref{sec:tridiagonalfrozen} contains a section where the finite variation terms are worked out and somewhat simplified. This explicit computation will be necessary for the asymptotic results in the next section, but is not essential otherwise. 

The final section gives the proof of Theorem \ref{thm:asymptotics}. A more detailed description about the organization may be found at the start of this section.

\section{Isospectral processes}
\label{sec:iso}

We begin by characterizing real matrix diffusions that carry the same eigenvalue distributions.  Suppose that $Z_t$ is a real symmetric semimartingale and $O_t$ is a real orthogonal semimartingale that are adapted to a common filtration.  Real orthogonal processes can be characterized by an associated process on the lie algebra associated to the orthogonal group, the real antisymmetric matrices.  Specifically, if $M_t$ is any real antisymmetric continuous semimartingale, the process
\[
  dO_t^t = O_t^t( dM_t + \frac{1}{2} (dM_t)^2)
\]
is a real symmetric process.  Moreover, any real orthogonal continuous semimartingale arises in this way (see the proof sketch in Theorem~\ref{thm:iso}).

Let $Y_t=O_t Z_t O_t^t.$  Applying the stochastic product rule to $Y_t,$ we get
\begin{align*}
  dY_t &= 
  dO_t Z_t O_t^t
  + O_t dZ_t O_t^t
  + O_t Z_t dO_t^t \\
  &+ dO_t dZ_t O_t^t
  + dO_t Z_t dO_t^t
  + O_t dZ_t dO_t^t.
\end{align*}
Let $W_t = \int_0^t O_s dZ_s O_s^t.$  Then, we can write
\begin{align*}
  dY_t &= 
  (dM_t + \frac{1}{2} (dM_t)^2)^t Y_t 
  + dW_t  + Y_t( dM_t + \frac{1}{2} (dM_t)^2) \\
  &+ dM_t^t dW_t
  +dM_t^t Y_tdM_t
  +dW_t dM_t,
\end{align*}
noting that higher order terms disappear.  This can be rewritten as
\[
  dY_t = dW_t + [Y_t, dM_t] + \frac{1}{2}[2dW_t + [Y_t,dM_t], dM_t].
\]
Alternatively, using that $[Y_t,dM_t] = dY_t - dW_t,$ up to finite variation terms,
\[
  dY_t = dW_t + [Y_t, dM_t] + \frac{1}{2}[dW_t + dY_t, dM_t].
\]
We have, in effect, proven the following.
\begin{theorem}
  Suppose $\MP{Z}$ is a real symmetric process adapted to a filtration $\filt$ and $\MP{Y}=\MP{O}\MP{Z}\MP{O}^t$ for some continuous orthogonal process $\MP{O}$ adapted to $\filt,$ then
  \begin{equation}
    \label{eq:prog}
    \begin{aligned}
    dY_t &= dW_t + [Y_t, dM_t] + \frac{1}{2}[dW_t + dY_t, dM_t] \\
    dO_t^t &= O_t^t( dM_t + \frac{1}{2} (dM_t)^2)
    \end{aligned}
  \end{equation}
  for some real anti-symmetric process $\MP{M}$ adapted to $\filt$ and where $\MP{W}$ solves
  $dW_t = O_t dZ_t O_t^t$ for all $t \geq 0.$
  \label{thm:iso}
\end{theorem}
\begin{proof}
  We need to demonstrate the existence of the process $\MP{M}$ solving
  \[
    dO_t^t = O_t^t( dM_t + \frac{1}{2} (dM_t)^2).
  \]
  On having done so, the product rule computations preceeding the theorem statement completes the proof.
  Let $T > 0$ be arbitrary.
  For any fixed time $t_0 \in [0,T],$ we have that 
  \(
    \lim_{t \to t_0} O_t^t O_{t_0} = \operatorname{Id}.
  \)
  Hence, we can take the matrix logarithm of $\tilde M_{t_0,t} = \log(O_t^t O_{t_{0}})$ for time $t$ sufficiently close to $t_0,$ which as an analytic function of $O_t$ is adapted to $(\filt_t, t \geq t_0).$  Further, we have that 
  \(
  O_t^t = O_{t_0}^t e^{\tilde M_{t_0,t}}
  \)
  for a sufficiently small window of time around $t_0.$  Hence, for a mesh $\left\{ t_i \right\}_0^d$ of $[0,T]$ of sufficiently fine spacing we may define for $t_i \leq t < t_{i+1}$
  \(
  M_t^d = \tilde M_{t_i,t} + \sum_{j=1}^i \tilde M_{t_{j-1},t_{j}}.
  \)
  On sending the mesh size to $0$ it is routine to check that $(M_t^{(d)}, 0 \leq t \leq T)$ converges to a solution of 
  \(
    dO_t^t = O_t^t( dM_t + \frac{1}{2} (dM_t)^2).
  \)

\end{proof}

\section{The time evolving tridiagonal model for general $\beta$}
\label{sec:dysoniso}

So, any time evolving tridiagonal model for which the eigenvector process is adapted will follow \eqref{eq:prog}.  In particular, any adapted conjugation of $\beta$-Dyson Brownian motion, which is a real diagonal semimartingale, has this form.  Since the eigenvalues and eigenvectors form locally differentiable processes in terms of the matrix entries, it follows that any tridiagonal model for $\beta$-Dyson Brownian motion satisfies \eqref{eq:prog}.
\begin{theorem}
\label{thm:dysoniso}
  If $\MP{A}$ is a tridiagonal model for $\beta$-Dyson Brownian motion with $\beta \geq 1$ whose starting configuration is almost surely simple, i.e.\;there is an orthogonal matrix process $\MP{O}$ so that $\MP{A} = \MP{O} \MP{\Lambda} \MP{O}^t$ for diagonal $\MP{\Lambda}$ satisfying for each $1 \leq i \leq n,$
  \[
    d\Lambda_{i,t} = \sqrt{\frac{2}{\beta}}dB_{i,t} - \biggl( \frac{V'(\lambda_i)}{2} - \sum_{j: j \neq i} \frac{1}{\lambda_i - \lambda_j} \biggr)\,dt
  \]
  then with respect to the filtration $(\filt_t, t\geq 0) = (\sigma(A_t),t\geq 0),$
  \[
    \begin{aligned}
    dA_t &= dW_t + [A_t, dM_t] + \frac{1}{2}[dW_t + dA_t, dM_t] \\
    dO_t^t &= O_t^t( dM_t + \frac{1}{2} (dM_t)^2)
  \end{aligned}
  \]
  for some real antisymmetric process $\MP{M}$ where $\MP{W}$ solves $dW_t = O_t d\Lambda_t O_t^t.$ 
  \label{thm:tri}
\end{theorem}
\begin{proof}
  Since the eigenvalues of $\MP{A}$ are always distinct almost surely, the eigenvector matrix $\MP{O}$ is continuous and in fact differentiable as a function of the entries of $\MP{A}.$  Hence $\MP{O}$ is a continuous semimartingale adapted to $\filt$ and by Theorem~\ref{thm:iso} the result follows.  
 \end{proof}

The presence of the Dyson Brownian motion generator inside the theorem may at first sight appear disconcerting.  For example, we may wish to find a representation for the driving term $dW_t$ so that it is a local martingale.  Indeed this is possible:
\begin{theorem}
  Suppose that $\MP{W}$ is a symmetric matrix diffusion that solves
  \[
    dW_t = \sqrt{\tfrac{2}{\beta}} O_t dB_t O_t^t + O_t dZ_t O_t^t,
  \]
  where $\beta \geq 1,$ $\MP{B}$ is a diagonal matrix Brownian motion, $\MP{Z}$ is a symmetric matrix Brownian motion with $0$-diagonal and $\MP{O}$ is a continuous version of the eigenvector matrix of $\MP{W}.$  Then, $\MP{\Lambda} = \MP{O}^t \MP{W} \MP{O}$ has entries satisfying
  \begin{equation}
    \label{eq:dl1}
    d\Lambda_{i,t} = \sqrt{\frac{2}{\beta}}dB_{i,t} +\sum_{j: j \neq i} \frac{1}{\lambda_i - \lambda_j}\,dt,
  \end{equation}
  that is the eigenvalues of $\MP{W}$ follow $\beta$-Dyson Brownian motion with no potential.
  Hence if $\MP{A} = \MP{U}\MP{W}\MP{U}^t$ is any tridiagonal matrix diffusion satisfying
  \[
    \begin{aligned}
    dA_t &= U_tdW_tU_t^t + [A_t, dM_t] + \frac{1}{2}[dW_t + dA_t, dM_t] \\
    dU_t^t &= U_t^t( dM_t + \frac{1}{2} (dM_t)^2),
    \end{aligned}
  \]
  then the eigenvalues of $\MP{A}$ evolve according to \eqref{eq:dl1}.

  \label{thm:tri2}
\end{theorem}
We require some linear algebra before proving this theorem, which we give in Section~\ref{sec:commutator}
However, we note that in the case $\beta=1,$ this driving noise simplifies:
\begin{corollary}
  In the case $\beta=1,$ the driving noise $\MP{W} \lawequals \frac{\MP{G} + \MP{G}^t}{\sqrt{2}}$ where $\MP{G}$ is a matrix of i.i.d.\;standard Brownian motions.
\end{corollary}
\noindent The previous Theorem and Corollary are related to work of Allez and Guionnet see \cite{AllezGuionnet}.

\begin{remark}
  Similar simplifications occur for $\beta=2$ and $\beta=4$ if we replace the process $\MP{O}$ on the orthogonal group by appropriate unitary or symplectic processes, respectively. 
\end{remark}


\section{Solving the commutator equation}
\label{sec:commutator}

Introduce the space $\TS$ of symmetric tridiagonal matrices.  
In this section, we consider the linear algebra problem of solving for antisymmetric $M$ so that 
\begin{equation}
  [A,M] + W \in \TS.
  \label{eq:comm}
\end{equation}
By dimension counting, we can see that this problem does not have unique solutions.  Indeed there is possibly (and generically there is) an $(n-1)$-dimensional space of $M$ that produce the desired answer (see Theorem \ref{thm:span}).
This is locally the problem that must be solved to describe the rotation differentials $dM_t$ that tridiagonalize $A_t$ when perturbed by $dW_t$ in \eqref{eq:prog}.  Further, we will need to understand the resulting space of tridiagonal matrices that solve this equation for a given $W.$
Understanding this algebra leads us to the theory of discrete orthogonal polynomials.

\subsection{Discrete orthogonal polynomials}
\label{sec:OPs}

Let
\begin{equation}
A = \left[ \begin{array}{ccccc}
b_1 & a_1 & && \\
a_1 & b_2 & a_2 & &\\
& \ddots & \ddots & \ddots \\
& & a_{n-2} & b_{n-1} & a_{n-1}\\
&&& a_{n-1} & b_n
\end{array}\right]
\end{equation}
be a Jacobi matrix. We can define a sequence of polynomials $p_0(x), p_1(x),...,p_n(x)$ to be solutions to the three term recurrence generated by the equation $A\Bp(x)=x\Bp(x)$ with $\Bp^T(x)=[p_0(x),p_1(x),...,p_{n-1}(x)]$. Taking $p_0(x)=1$ we get 
\begin{align*}
b_1 p_0(x)+ a_1 p_1(x) & =x p_0(x)\\
a_{k-1}p_{k-2}(x) + b_k p_{k-1}(x) + a_k p_k(x) & =x p_{k-1}(x)\\
a_{n-1} p_{n-2}(x) + (b_n-x)p_{n-1}(x) & = - p_n(x).
\end{align*}
For convenience we have completed the recurrence by taking $a_{n} =1$. This makes $p_n(x)$ a multiple of $\det(x-A).$  If $\lambda$ is an eigenvalue of $A$ the recurrence closes exactly and so $p_n(\lambda) = 0$. In other words the zeros $\lambda_1,...,\lambda_n$ of $p_n(x)$ are the eigenvalues of $A$, and moreover they have associated eigenvectors $\Bp(\lambda_k)$. Because $A$ is self-adjoint it follows that the eigenvectors are orthogonal. Take $q_k^{-2} = \| \Bp(\lambda_k)\|^2$, then the matrix 
\begin{equation}
  \label{eq:COB}
O^t= \left[ \begin{array}{cccc}
q_1 & q_1 p_1(\lambda_1) & \cdots & q_1 p_{n-1}(\lambda_1)\\
q_2 & q_2 p_1(\lambda_2) & \cdots & q_2 p_{n-1}(\lambda_2)\\
\vdots & & & \vdots \\
q_n & q_n p_1(\lambda_n) & \cdots & q_n p_{n-1}(\lambda_n)
\end{array}
\right]
\end{equation}
is the orthogonal change of basis matrix that diagonalizes $A$ with $O^t A O = \Lambda$. Moreover because $O$ is orthogonal we have the standard orthonormality conditions on both rows and columns. Therefore we get the following relationships:
\begin{align}
\sum_{i=1}^n q_i^2 p_{k}(\lambda_i) p_{\ell}(\lambda_i) = \delta_{k,\ell},\\
\sum_{i=1}^n q_k q_\ell p_{i-1}(\lambda_k)p_{i-1}(\lambda_\ell) = \delta_{k,\ell}. 
\end{align}

\subsection{Christoffel-Darboux Formula}

The Christoffel-Darboux formula is an automatic consequence of the three-term recurrence for $\left\{ p_k \right\}_{k=0}^n.$  As it is in symmetric form, this formula is
\begin{equation}
  \sum_{k=0}^m p_k(x)p_k(y) = a_{m+1}\frac{p_{m+1}(x)p_m(y) - p_m(x)p_{m+1}(y)}{x-y}.
  \label{eq:CD}
\end{equation}
Multiplying the left hand side through by $(x-y),$ and applying the recurrence, all but the terms on the right survive.  We can also take the limit in this formula as $y \to x$ to get the \emph{confluent} form of this formula:
\begin{equation}
  \sum_{k=0}^m p_k(x)^2 = a_{m+1}(p_{m+1}'(x)p_m(x) - p_m'(x)p_{m+1}(x) ).
  \label{eq:cCD}
\end{equation}
This allows us to represent the spectral weights $q_i^2$ by the following special case
\begin{equation}
  \frac{1}{q_i^2} = \sum_{k=0}^{n-1} p_k(\lambda_i)^2 = p_{n}'(\lambda_i)p_{n-1}(\lambda_i).
  \label{eq:qipn}
\end{equation}
We will also use the summation
\begin{align}
  \sum_{k=0}^m p_k(x)p'_k(y) 
  = &a_{m+1}\partial_y\biggl(\frac{p_{m+1}(x)p_m(y) - p_m(x)p_{m+1}(y)}{x-y}\biggr) \nonumber\\
  = 
  &a_{m+1}\frac{p_{m+1}(x)p'_m(y) - p_m(x)p'_{m+1}(y)}{x-y} \nonumber \\
  + &a_{m+1}\frac{p_{m+1}(x)p_m(y) - p_m(x)p_{m+1}(y)}{(x-y)^2}. 
  \label{eq:dCD}
\end{align}
Sending $x$ to $y$ (or more directly, differentiating \eqref{eq:cCD}), we also get
\begin{equation}
  \sum_{k=0}^m p_k(x)p_k'(x) = a_{m+1}(p_{m+1}''(x)p_m(x) - p_m''(x)p_{m+1}(x) ).
  \label{eq:dcCD}
\end{equation}


\subsection{Fundamental identity}

The commutator at entry $k,\ell$ is given by 
\begin{equation}
[A,M]_{k,\ell} = a_{k-1} m_{k-1,\ell} + b_k m_{k,\ell} +a_k m_{k+1,\ell} - (a_{\ell-1} m_{k,\ell-1} + b_\ell m_{k,\ell}+ a_\ell m_{k,\ell+1}).
\label{eq:AM}
\end{equation}
Terms for which $k$ or $\ell$ exceed the allowable indices are taken to be $0.$
This makes a natural choice for $M$ to be built out of the orthogonal polynomials associated to $A$, $p_{k-1}(x)p_{\ell-1}(y).$
To this end, define the matrix
\[
  E(x,y)_{k,\ell} = 
  p_{k-1}(x) p_{\ell-1}(y) \one[ k > \ell]
  +\frac{1}{2}p_{k-1}(x) p_{k-1}(y) \one[ k = \ell].
\]
In terms of $E,$ let $E^+$ be the symmetric analogue and let $E^{-}$ be the antisymmetric analogue 
\begin{equation}
\label{eq:E}
E^{+}(x,y) = E(x,y) + E(x,y)^t \qquad \text{ and } \qquad
E^{-}(x,y) = E(x,y) - E(x,y)^t.
\end{equation}

Some of these matrices, when evaluated at the eigenvalues of $A,$ have a nice expression when conjugated by the eigenvector change of basis matrix $O.$ These changes of basis are simple to check, and we do not verify them.  In what follows, we let $\delta_{i,j}$ be the matrix which zero everywhere except in the $(i,j)$ position where it is $1.$
\begin{align}
  &E^{+}(\lambda_i,\lambda_j) 
  + 
  E^{+}(\lambda_j,\lambda_i) 
  =
  O^t \biggl( \frac{\delta_{i,j} + \delta_{j,i}}{q_iq_j} \biggr) O,
  \label{eq:spec1} \\
  &E^{+}(\lambda_i,\lambda_i) 
  =
  O^t\biggl(
  \frac{\delta_{i,i}}{q_i^2}
  \biggr) O,
  \label{eq:spec2} \\
  &E^{-}(\lambda_i,\lambda_j) 
  -
  E^{-}(\lambda_j,\lambda_i) 
  =
  O^t \biggl( \frac{\delta_{i,j} - \delta_{j,i}}{q_iq_j} \biggr) O.
  \label{eq:spec3}
\end{align}
We note that \eqref{eq:spec2} is really just a special case of \eqref{eq:spec1}, but we include it for emphasis.  Note that, despite the importance of $E^{-}(\lambda_i,\lambda_i),$ which will become clear, it does not have any such simple representation.

Commutation of $A$ with $E^{-}$ nearly gives a multiple of $E^+.$  In effect, save for boundary terms that result from the problem having finite size and the boundary terms that result from the necessary antisymmetry conditions, $[A,E^{-}(x,y)] = (x-y)E^{+}:$
\begin{proposition}
  Define the matrix $\tilde B(x,y) = [A,E^{-}(x,y)] - (x-y)E^{+}(x,y).$  Then, for any $k \geq \ell,$
  \[
    \tilde B(x,y)_{k,\ell}
    =\begin{cases}
      -a_{\ell-1} p_{\ell-2}(x)p_{\ell-1}(y) + a_\ell p_{\ell-1}(x) p_{\ell}(y)&  \\
      -a_{\ell-1} p_{\ell-1}(x)p_{\ell-2}(y) + a_\ell p_{\ell}(x) p_{\ell-1}(y),  & \text{ if } k = \ell  < n, \\
      -a_\ell (p_{\ell-1}(x)p_{\ell-1}(y) - p_\ell(x)p_\ell(y)), & \text{ if } k = \ell + 1 < n,\\
      -p_n(x)p_{\ell-1}(y), & \text{ if } k = n, \ell < n-1, \\
      -p_n(x)p_{\ell-1}(y) 
      -a_\ell(p_{\ell-1}(x)p_{\ell-1}(y) - p_\ell(x)p_\ell(y))
      , & \text{ if } k = n, {\ell = n-1},\\
      -a_{\ell-1} p_{\ell-2}(x)p_{\ell-1}(y) + a_\ell p_{\ell-1}(x) p_{\ell}(y)&  \\
      -a_{\ell-1} p_{\ell-1}(x)p_{\ell-2}(y) - a_\ell p_{\ell}(x) p_{\ell-1}(y),
      & \text{ if } k = n, {\ell = n}, \text{ or } \\
      0, & \text{ otherwise.}
    \end{cases}
  \]
  Observe that the matrix $[A,E^{-}(x,y)] - (x-y)E^{+}(x,y)$ is symmetric, which determines the $k > \ell$ terms. 
  \label{prop:commutator}
\end{proposition}
\begin{proof}
  \noindent If $1 < k < n$ and $\ell + 1 < k$ there are no boundary effects.  Hence \eqref{eq:AM} becomes
  \begin{align*}
    [A,E^{-}(x,y)]_{k,\ell} = 
    &(a_{k-1} p_{k-2}(x) + b_k p_{k-1}(x) + a_k p_{k}(x))p_{\ell-1}(y) - \\
    &p_{k-1}(x)(a_{\ell-1} p_{\ell-2}(y) + b_\ell p_{\ell-1}(y) + a_\ell p_{\ell}(y)) \\
    =&(x-y)p_{k-1}(x)p_{\ell-1}(y).
  \end{align*}
  The formula additionally holds when $k = 1,$ as the orthogonal polynomial recurrence for this case holds with $a_0 = 0.$

  \vspace{1cm}
  \noindent $k = n:$  
  When $k = n$ and $\ell + 1 < n,$ due to the[A,M] + W border at the bottom of the matrix, we have
  \begin{align*}
    [A,E^{-}(x,y)]_{n,\ell} = 
    &(a_{n-1} p_{n-2}(x) + b_n p_{n-1}(x) + 0 \cdot p_{n}(x))p_{\ell-1}(y) - \\
    &p_{n-1}(x)(a_{\ell-1} p_{\ell-2}(y) + b_\ell p_{\ell-1}(y) + a_\ell p_{\ell}(y)) \\
    =&(x-y)p_{n-1}(x)p_{\ell-1}(y) - p_n(x) p_{\ell-1}(y),
  \end{align*}
  using that $a_n$ is taken to be $1.$  

  \vspace{1cm}
  \noindent $k = \ell+1:$  
  When $k = \ell+1$ and $\ell + 1 < n,$ due to the diagonal of $E^{-}$ being $0,$ we have
  \begin{align*}
    [A,E^{-}(x,y)]_{\ell+1,\ell} = 
    &(0 \cdot p_{\ell-1}(x) + b_{\ell+1} p_{\ell}(x) + a_{\ell+1} p_{\ell+1}(x))p_{\ell-1}(y) - \\
    &p_{\ell}(x)(a_{\ell-1} p_{\ell-2}(y) + b_\ell p_{\ell-1}(y) + 0 \cdot p_{\ell}(y)) \\
    =&(x-y)p_{\ell}(x)p_{\ell-1}(y) - a_\ell (p_{\ell-1}(x)p_{\ell-1}(y) - p_\ell(x)p_\ell(y)).
  \end{align*}
  If $\ell+1 = n,$ then we must subtract from the expression above for $[A,E^{-}(x,y)]_{n,n-1}$ any terms containing $a_n.$

  \vspace{1cm}
  \noindent $k = \ell:$ 
  When $k=\ell,$ antisymmetry of $E^{-}$ plays a larger role.  Taking care that the terms above the diagonal are properly accounted, for $k=\ell < n,$ 
  \begin{align*}
    [A,E^{-}(x,y)]_{\ell,\ell} = 
    &(0 \cdot p_{\ell-2}(x) + 0 \cdot p_{\ell-1}(x) + 2 a_\ell p_{\ell}(x))p_{\ell-1}(y) - \\
    &p_{\ell-1}(x)(2a_{\ell-1} p_{\ell-2}(y) + 0 \cdot p_{\ell-1}(y) + 0 \cdot p_{\ell}(y)) \\
    =&(x-y)p_{\ell-1}(x)p_{\ell-1}(y) 
    - a_{\ell-1} p_{\ell-2}(x)p_{\ell-1}(y) 
    + a_{\ell} p_{\ell}(x)p_{\ell-1}(y) \\
    -&a_{\ell-1} p_{\ell-1}(x) p_{\ell-2}(y)
    + a_\ell p_{\ell-1}(x) p_{\ell}(y). 
  \end{align*}
  If $k=\ell = n,$ then we must subtract from the expression above for $[A,E^{-}(x,y)]_{n,n-1}$ any terms containing $a_n.$
\end{proof}

In the case that we pick $x=\lambda_i$ and $y=\lambda_j$ to be eigenvalues of $A,$ the error term $\tilde B(\lambda_i,\lambda_j)$ will be tridiagonal.  We begin by characterizing the space of matrices whose commutator with $A$ is tridiagonal.  Recall that $\left\{ \lambda_i : 1 \leq i \leq n \right\}$ are the eigenvalues of $A.$  Essentially everything we do here is predicated on working with tridiagonal matrices $A$ which have nonzero offdiagonal entries.  We call these tridiagonal matrices \emph{nondegenerate}.

\begin{theorem}
  Suppose that $A$ is \emph{nondegenerate} and all $\left\{ \lambda_i \right\}$ are distinct.
  Let $V \subset \mathbb{M}_{n}$ be the space of antisymmetric matrices so that for all $v \in V,$
  \[
    [A,v] \in \TS.
  \]
  The space $V$ is spanned by $\left\{ E^{-}(\lambda_i,\lambda_i) : 1 \leq i \leq n\right\}.$ 
  \label{thm:span}
\end{theorem}
\begin{proof}
  The argument is by dimension counting.  We start by showing that the dimension of $V$ is $n-1.$  Observe first that there is a linear dependence among the $V$ that follows from the orthogonality of the polynomials. For any $\ell > r$ in $[n],$
  \[
    \sum_{i=1}^n q_i^2 E^{-}(\lambda_i,\lambda_i)_{\ell,r} 
    = 
    \sum_{i=1}^n q_i^2 p_{\ell-1}(\lambda_i) p_{r-1}(\lambda_i) = 0. 
  \]
  As the diagonal of $E^{-}$ is identically $0,$ this linear combination of these matrices is identically $0.$  On the other hand, looking at the first columns of $E^{-}(\lambda_i,\lambda_i),$ the span of these first columns is at least $n-1$ dimensional from the linear independence of $\left\{ p_k : k=1,2,\dots,n-1 \right\}.$

  On the other hand, by the lattice path solution, after picking the first column of $v,$ the constraint that $[A,v] \in \TS$ determines the remainder of $v$ when $A$ has nonzero offdiagonal entries.  Hence, this space is at most $(n-1)$-dimensional, which completes the proof.
\end{proof}

\begin{remark}
  The requirement that $A$ have nonzero offdiagonal entries is necessary for the previous theorem.  When $A$ is diagonal for example, the commutator of any tridiagonal matrix with $A$ will be tridiagonal.
\end{remark}

\subsection{Symmetrized form}

Define the symmetrized matrix
\[
  H^{-}(x,y)
  = \frac{E^{-}(x,y) - E^{-}(y,x)}{x-y}.
\]
Define $L(x,y) = [A,H^{-}(x,y)] - E^{+}(x,y) - E^{+}(y,x).$
By Proposition~\ref{prop:commutator} we have that
\begin{equation}
  L(x,y)_{k,\ell}
    =\begin{cases}
      \frac{-p_n(x)p_{\ell-1}(y) + p_{\ell-1}(x)p_n(y)}{x-y}, & \text{ if } k = n, \ell {\leq} n-1, \\
      \frac{-2p_n(x)p_{n-1}(y) + 2p_{n-1}(x)p_n(y)}{x-y}, & \text{ if } k = n, \ell = n, \\
      0, & \text{ otherwise.}
    \end{cases}
  \label{eq:L}
\end{equation}
Observe that this error vanishes entirely when $x$ and $y$ are chosen to be eigenvalues of $A.$  In short, this is because the matrix $H^{-}$ has a simple expression when conjugated by $O$ (recall \eqref{eq:COB} and \eqref{eq:spec3})
\begin{equation*}
  O^t H^{-}(\lambda_i, \lambda_j) O = \frac{\delta_{i,j} + \delta_{j,i}}{q_iq_j(\lambda_i - \lambda_j)},
\end{equation*}

with $\delta_{i,j}$ the matrix with a single nonzero entry, equal to $1,$ in position $(i,j).$
This gives an alternative derivation of \eqref{eq:L} for the case that $x$ and $y$ are eigenvalues by writing (and recalling \eqref{eq:spec1})
\begin{equation}
  \label{eq:H-}
  [A,H^{-}(\lambda_i,\lambda_j)]
  = O\biggl[\Lambda, \frac{\delta_{i,j} - \delta_{j,i}}{q_iq_j(\lambda_i - \lambda_j)}\biggr]O^t
  = O\biggl(\frac{\delta_{i,j} + \delta_{j,i}}{q_iq_j}\biggr)O^{t}
  = {E^{+}(\lambda_i,\lambda_j) + E^{+}(\lambda_j,\lambda_i)}.
\end{equation}

This will prove useful in what follows, but it will also be helpful to define another preimage of $E^{+}(\lambda_i,\lambda_j) + E^{+}(\lambda_j,\lambda_i) + \TS$ that vanishes in the first column.
The matrix $H^{-}$ in its first column has $k$-th entry
  \[
    H^{-}(\lambda_i, \lambda_j)_{k,1}
    =\frac{p_{k-1}(\lambda_i) - p_{k-1}(\lambda_j)}{\lambda_i - \lambda_j}.
  \]
  This is equal to the $k$-th entry of the first column of the matrix 
  \begin{equation*}
    \frac{E^{-}(\lambda_i,\lambda_i) - E^{-}(\lambda_j,\lambda_j)}{\lambda_i-\lambda_j}.
  \end{equation*}
Hence, we define
\begin{equation}
H_0^{-}(x,y)
= \frac{E^{-}(x,y) - E^{-}(y,x)-E^{-}(x,x) + E^{-}(y,y)}{x-y},
\label{eq:H0}
\end{equation}
which satisfies $[A,H_0^{-}(\lambda_i,\lambda_j)] - E^{+}(\lambda_i,\lambda_j) - E^{+}(\lambda_j,\lambda_i) \in \TS$ and $H_0^{-} e_1=0,$ with $e_1$ the first standard basis vector.

Recapitulating: we would like to solve the equation 
\[
  [A,M] + W \in \TS
\]
where $W$ is an arbitrary symmetric matrix for $M$.  If $W$ is in the span of 
\[
\left\{ {E^{+}(\lambda_i,\lambda_j) + E^{+}(\lambda_j,\lambda_i), i < j} \right\},
\]
then we can do this exactly without the need for a tridiagonal error.  Thus it essentially remains to determine how to solve for 
\[
  [A,M] + E^{+}(\lambda_i,\lambda_i) \in \TS
\]
for each $1\leq i \leq n.$
Noting that $O^t E^{+}(\lambda_i,\lambda_i) O = q_i^{-2} \delta_{i,i},$ the combination of these matrices and 
\[
  \left\{ {E^{+}(\lambda_i,\lambda_j) + E^{+}(\lambda_j,\lambda_i)}, i < j \right\},
\]
clearly span all symmetric matrices, and so we will have a complete solution.

Using this discussion, we can give a quick proof of Theorems \ref{thm:tri2}.
\begin{proof}[Proof of Theorem~\ref{thm:tri2}]
  We recall the setup for the theorem.  Suppose that $\MP{W}$ is a stochastic matrix diffusion that satisfies 
  \[
  d W_t = \sqrt{\tfrac{2}{\beta}} O_t dB_t O_t^t + O_t dZ_t O_t,
  \] 
  where $\MP{O}$ is the eigenvector matrix of $\MP{W}$. Let $\MP{\Lambda}$ be the eigenvalue matrix of $\MP{W}$, we check that 
  \begin{align*}
  d \Lambda_t &= O_t^t dW_t O_t + [\Lambda_t, dN_t] + \frac{1}{2} [ O_t^t dW_t O_t + d\Lambda_t, dN_t]\\
  & =\sqrt{\tfrac{2}{\beta}}dB_t+dZ_t + [\Lambda_t, dN_t] + \frac{1}{2} [ \sqrt{\tfrac{2}{\beta}}dB_t+dZ_t+ d\Lambda_t, dN_t]
  \end{align*}
  where $N_t$ is chosen so that $dN_{i,j,t}(\lambda_i-\lambda_j) = dZ_{i,j,t}$ which gives perfect cancelation of the $dZ_t$ term. This induces the evolution of $\MP{O}$ and so the evolution of $\MP{W}$. By independence the previous equation simplifies further to 
  \[
  d \Lambda_t =\sqrt{\tfrac{2}{\beta}}dB_t+ \frac{1}{2} [ dZ_t, dN_t].
  \]
  Further
    \[
    [dZ_t, dN_t ]_{i,i}
    = \sum_{j \neq i} \frac{
      dZ_{i,j,t}dZ_{j,i,t}
    }{\lambda_{j,t}-\lambda_{i,t}}
    -\sum_{j \neq i} \frac{
      dZ_{i,j,t}dZ_{j,i,t}
    }{\lambda_{i,t}-\lambda_{j,t}}
    = \sum_{j \neq i} \frac{
    -2dt}{\lambda_{i,t}-\lambda_{j,t}}.
  \]
  This gives us that $\Lambda$ evolves according to a Dyson Brownian motion with zero potential. To get the conclusion of the theorem we can take any tridiagonal evolution of $\MP{W}$ and the diffusion equation follows from Theorem \ref{thm:dysoniso}.
\end{proof}

\subsection{Confluent form}

Going forward, we will essentially always specialize $x$ and $y$ to be eigenvalues.  For eigenvalues the matrix $B$ simplifies.  So, we define
\begin{equation}
    B(x,y)_{k,\ell}
    =\begin{cases}
      -a_{\ell-1} p_{\ell-2}(x)p_{\ell-1}(y) + a_\ell p_{\ell-1}(x) p_{\ell}(y)&  \\
      -a_{\ell-1} p_{\ell-1}(x)p_{\ell-2}(y) + a_\ell p_{\ell}(x) p_{\ell-1}(y),  & \text{ if } k = \ell, \\
      -a_\ell (p_{\ell-1}(x)p_{\ell-1}(y) - p_\ell(x)p_\ell(y)), & \text{ if } k = \ell + 1,\\
      0, & \text{ otherwise,}
    \end{cases}
    \label{eq:B}
\end{equation}
and observe that for eigenvalues $\lambda_1$ and $\lambda_2,$ $B(\lambda_1,\lambda_2) = \tilde B(\lambda_1,\lambda_2).$

One possible choice for $x$ and $y$ is to take them equal.  Indeed, by symmetrizing the $E^{-},$ dividing by $x-y$ and sending $x \to y$ we are also led to the existence of an approximate preimage to $E^{+}(y,y)$ that involves derivatives of $E^{-}.$  A more direct identity is possible simply by differentiating $E^{-}.$  Observe that
\[
  \partial_y [A,E^{-}(x,y)] = [A, \partial_y E^{-}(x,y)],
\]
by virtue of $A$ having no $y$ dependence.  Hence,
\[
  [A, \partial_y E^-(x,y)] 
  = \partial_y [A, E^-(x,y)]
  = -E^+(x,y) + (x-y)\partial_y E^+(x,y) + \partial_y\tilde B(x,y).
\]
This operation can be iterated to produce a basis of derivatives, by using
\[
  [A, \partial_x^k \partial_y^\ell E^{-}(x,y)]
  =\partial_x^k \partial_y^\ell \left( (x-y)E^{+}(x,y) + \tilde B(x,y) \right),
\]
but we do not pursue this further.
If we define $F(y) = -\partial_y E^{-}(x,y) \vert_{x=y},$ then
\begin{equation}
  [A,F(y)] = E^{+}(y,y) - \partial_y \tilde B(x,y) \vert_{x=y}.
  \label{eq:F}
\end{equation}
We write some entries of this error term:
\[
  -\partial_y \tilde B(x,y) \vert_{x=y}
  =\begin{cases}
    a_{\ell-1} p_{\ell-2}(y)p_{\ell-1}'(y) - a_\ell p_{\ell-1}(y) p_{\ell}'(y)&  \\
    +a_{\ell-1} p_{\ell-1}(y)p_{\ell-2}'(y) - a_\ell p_{\ell}(y) p_{\ell-1}'(y),  & \text{ if } k = \ell  < n, \\
    a_\ell (p_{\ell-1}(y)p_{\ell-1}'(y) - p_\ell(y)p_\ell'(y)), & \text{ if } k = \ell + 1 < n,\\
    p_n(y)p_{\ell-1}'(y), & \text{ if } k = n, \ell < n-1, \\
    \text{other}, & \text{ if } k = n, {\ell \geq n-1}, \text{ or } \\
    0, & \text{ otherwise.}
  \end{cases}
\]

Specializing to $x=y=\lambda_i$ for a zero $\lambda_i $ of $p_n$ will be especially useful here, as in this case the boundary term above is tridiagonal. In addition, in this case, the formulas for the tridiagonal that hold in the middle of the matrix also hold in the final entries.  Otherwise said, at an eigenvalue $\lambda_i$ we define
\begin{equation}
  \label{eq:G}
  G(\lambda_i)
  =
  -\partial_y \tilde B(x,y) \vert_{x=y=\lambda_i}
  = 
  -\partial_y B(x,y)\vert_{x=y=\lambda_i} 
  = 
  -\frac{1}{2}\partial_y B(y,y)\vert_{y=\lambda_i}, 
\end{equation}
and we define $G(y) = -\frac{1}{2}\partial_y B(y,y)$ for general $y.$


\subsection{Difference quotients and minor polynomials}
\label{sec:minors}

Define the difference quotient operator of a polynomial as
\[
  (\Delta^y f)(x) = \frac{f(x)-f(y)}{x-y}.
\]
We always consider $\Delta^y$ as a map on polynomials.  In particular, for $x=y,$ the previous definition extends by continuity.
For a multivariate polynomial, let $\Delta^y_x$ denote the corresponding partial difference quotient in the $x$ variable at $y.$

For any polynomial $p,$ the function  
\[
  \Delta^y p(x) = \frac{p(x)-p(y)}{x-y}
\]
is a polynomial in two variables, which as a polynomial in $x$ has degree one less than $p.$  By continuity, it takes the value $p'(x)$ at $y=x.$
With this in mind, we define the polynomials for any $0 \leq k,\ell \leq n-1,$
\begin{equation}
  p^{(k)}_{\ell}(x) = \sum_{j=1}^n q_j^2 
  p_k(\lambda_j)
  \frac{p_\ell(x)-p_\ell(\lambda_j)}{x-\lambda_j},
  \label{eq:minor}
\end{equation}
for $x$ not in the spectrum.  Observe that if $\ell \leq k,$ then by orthogonality, this polynomial is $0.$  On the other hand, observe that
\begin{equation*}
  \begin{aligned}
    xp^{(k)}_{\ell}(x) 
    &= 
    \sum_{j=1}^n q_j^2 
    p_k(\lambda_j)
    \frac{xp_\ell(x)-xp_\ell(\lambda_j)}{x-\lambda_j} \\
    &= 
    \sum_{j=1}^n q_j^2 
    p_k(\lambda_j)
    \biggl\{
      \frac{xp_\ell(x)-\lambda_j p_\ell(\lambda_j)}{x-\lambda_j} 
      -p_\ell(\lambda_j)
    \biggr\}
    \\
    &=
    a_{\ell+1}p^{(k)}_{\ell+1}(x) 
    +b_{\ell+1}p^{(k)}_{\ell}(x) 
    +a_{\ell}p^{(k)}_{\ell-1}(x) 
    -\delta_{k,\ell}.
  \end{aligned}
\end{equation*}
This is to say that these polynomials satisfy the same $3$-term recurrence as $\left\{ p_\ell \right\}_{\ell}$ but started from different initial conditions.  This also implies that 
\[
  \deg(p^{(k)}_{\ell}) = (\ell-k-1)_{+} \quad \text{and} \quad p_{\ell+1}^{(\ell)}(x) = \frac{1}{a_{\ell+1}}.
\]
Otherwise stated, they are multiples of the orthogonal polynomials associated to the lower-right-principal submatrix of $A$ of size $n-k-1.$  
Furthermore, we have the identity that
\[
  \Delta^y p(x)
  =
  \frac{p_\ell(x)-p_\ell(y)}{x-y}
  =
  \sum_{k = 1}^\ell p^{(k-1)}_{\ell}(x) p_{k-1}(y),
\]
which by continuity specializes to 
\[
  p_\ell'(x) = \sum_{k = 1}^\ell p^{(k-1)}_{\ell}(x) p_{k-1}(x).
\]

We also observe that
\begin{equation}
  \sum_{k = 1}^\ell p^{(k-1)}_{\ell}(x) p_{k-1}'(y)
  =
  \partial_y \Delta^y p_\ell(x)
  =
  \frac{(p_\ell(x)-p_\ell(y))-(x-y)p_\ell'(y)}{(x-y)^2}
  = \Delta^y \Delta^y p_\ell(x),
  \label{eq:secondderiv}
\end{equation}
which by continuity specializes to $\frac{1}{2}p_\ell''(y)$ when $x=y.$

\subsubsection*{Approximating the minor polynomials}

Because the minor polynomials satisfy the same recurrence as the original polynomials, it is possible to formulate a mixed Christoffel--Darboux formula using the two sets of polynomials, namely for $r < \ell$
\begin{equation}
  \sum_{k=1}^{\ell} p_{k-1}(x) p_{k-1}^{(r)}(y) (x-y) = 
  a_{\ell} \left( 
  p_{\ell}(x) p_{\ell-1}^{(r)}(y)
  -p_{\ell-1}(x) p_{\ell}^{(r)}(y) \right) + p_{r}(x),
  \label{eq:CDmixed}
\end{equation}
where the extra $p_r(x)$ results from the defect in the recurrence for the minor polynomial in the bottom term.  Note that on taking $x=y$ in this formula, we are left with
\begin{equation}
p_{\ell-1}(x) p_{\ell}^{(r)}(x)
-
p_{\ell}(x) p_{\ell-1}^{(r)}(x)
=\frac{p_{r}(x)}{a_\ell},
  \label{eq:Wronskian}
\end{equation}
which shows that $\left\{ ((p_\ell), (p_\ell^{(0)})), \ell=1,2,\dots \right\}$ forms a fundamental set of solutions to the three--term recurrence for any initial conditions.  In particular, 
\begin{equation}
  \label{eq:MinorIdentity}
  p^{(j)}_\ell(x) = 
  \begin{cases}
  p_j(x) p_\ell^{(0)}(x)
  - 
  p_\ell(x)p_j^{(0)}(x), & j \leq \ell, \\
  0, & \text{otherwise,}
\end{cases}
\end{equation}
observing that the $3$--term--recurrence is trivially satisfied, and that the initial conditions at $\ell=j$ and $\ell=j+1$ are satisfied.  
This identity also shows that $\{p_{\ell}^{(j)}\}$ satisfies a recurrence in $j$ as well, specifically that for $j < \ell$
\begin{equation}
  \label{eq:DualRecurrence}
  x p^{(j)}_\ell(x) = 
  a_{j+1}p^{(j+1)}_\ell(x) 
  +b_{j+1}p^{(j)}_\ell(x) 
  +a_{j}p^{(j-1)}_\ell(x). 
\end{equation}
This formula is also true for $j=0,$ taking $a_0 = 1$ and $p^{(-1)}_\ell(x) = p_\ell(x),$ as it is equivalent to \eqref{eq:MinorIdentity} with $j=1.$

\subsubsection*{Perturbation theory for orthogonal polynomials}

The dual recurrence \eqref{eq:DualRecurrence} is used in \cite{Geronimo} to create a perturbation theory for orthogonal polynomials, one which shows that when two families of orthogonal polynomials have similar coefficients, they can be compared.  We present this perturbation theory here. 
Suppose that $\{\hat p_j(x)\}$ are defined by
\begin{align*}
\hat b_1 \hat p_0(x)+  \hat a_1  \hat p_1(x) & =x  \hat p_0(x)\\
\hat a_{j} \hat p_{j-1}(x) +  \hat b_{j+1}  \hat p_{j}(x) +  \hat a_{j+1}  \hat p_{j+1}(x) & =x  \hat p_{j}(x),
\end{align*}
for $j = 1,2,3,\dots.$  If we multiply \eqref{eq:DualRecurrence} through by $\hat p_j$ and subtract the previous equation multiplied by  $p^{(j)}_\ell(x),$ we get
\[
  0 =
  (a_{j+1}p^{(j+1)}_\ell(x) 
  +b_{j+1}p^{(j)}_\ell(x) 
  +a_{j}p^{(j-1)}_\ell(x))\hat p_j(x)
  -
  (\hat a_{j} \hat p_{j-1}(x) +  \hat b_{j+1}  \hat p_{j}(x) +  \hat a_{j+1}  \hat p_{j+1}(x)
  )p^{(j)}_\ell(x). 
\]
Hence, if we sum this identity for $j=0,1,\dots,\ell-1,$ we get
\begin{align*}
  0
  &=
  \sum_{j=0}^{\ell-1} (b_{j+1} - \hat b_{j+1})p^{(j)}_\ell(x)\hat p_j(x) 
  +
  \sum_{j=0}^{\ell-2} (a_{j+1} - \hat a_{j+1})
  (
  p^{(j+1)}_\ell(x)\hat p_j(x) 
  +p^{(j)}_\ell(x)\hat p_{j+1}(x) 
  ) \\
  &+a_0 p^{(-1)}_\ell(x)\hat p_0(x)-\hat a_\ell \hat p_\ell(x) p_\ell^{(\ell-1)}(x).
\end{align*}
In summary,
\begin{equation}
  \begin{aligned}
  \frac{\hat a_\ell}{a_\ell} \hat p_\ell(x)
  &=
  p_\ell(x)\hat p_0(x) \\
  &+
  \sum_{j=1}^{\ell} (b_{j} - \hat b_{j})p^{(j-1)}_\ell(x)\hat p_{j-1}(x) 
  \\
  &+
  \sum_{j=1}^{\ell-1} (a_{j} - \hat a_{j})
  (
  p^{(j)}_\ell(x)\hat p_{j-1}(x) 
  +p^{(j-1)}_\ell(x)\hat p_{j}(x) 
  ). 
\end{aligned}
  \label{eq:Perturb}
\end{equation}
Here we take $\hat{p}_j(x) = p_{j+1}^{(0)}(x),$ so that $\hat a_j = a_{j+1},$ $\hat b_j = b_{j+1},$ and $\hat p_0(x) = \frac{1}{a_1}.$  Specializing these formulas, we get
\begin{equation}
  \begin{aligned}
     p^{(0)}_{\ell+1}(x)
  &=
  \frac{a_\ell}{a_{\ell+1}a_1}p_\ell(x) \\
  &+
  \frac{a_\ell}{a_{\ell+1}}\sum_{j=1}^{\ell} (b_{j} - b_{j+1})p^{(j-1)}_\ell(x)p_{j}^{(0)}(x) 
  \\
  &+
  \frac{a_\ell}{a_{\ell+1}}\sum_{j=1}^{\ell-1} (a_{j} - a_{j+1})
  (
  p^{(j)}_\ell(x)p_{j}^{(0)}(x) 
  +p^{(j-1)}_\ell(x)p_{j+1}^{(0)}(x) 
  ). 
\end{aligned}
  \label{eq:MinorPerturb}
\end{equation}

\subsection{Summary}
\label{sec:summary}

The summary of the previous discussion is that we have a complete description of the solutions of the commutator equation \eqref{eq:comm} in terms of the orthogonal polynomials associated to $A.$
\begin{theorem}
  A solution to the commutator equation
  \[
    [A,M] - W \in \TS
  \]
  for antisymmetric $M$ in terms of a symmetric $W$ is given by $M = M_1 + M_2$ where
  \begin{align*}
    M_1 &= \sum_{i=1}^n (O^t W O)_{i,i} q_i^2 F(\lambda_i), \text{ and} \\
    M_2 &= \sum_{i > j}^n (O^t W O)_{i,j}q_iq_j H_0^{-}(\lambda_i,\lambda_j).
  \end{align*}
  Moreover, this is the unique $M$ that vanishes in the first column.
  The tridiagonal error is given by
  \begin{equation*}
    \begin{aligned}
    [A,M]-W &= 
    \sum_{i=1}^n (O^t W O)_{i,i} q_i^2 G(\lambda_i)
    -\sum_{i > j}^n (O^t W O)_{i,j} q_iq_j
    \frac{B(\lambda_i,\lambda_i) - B(\lambda_j,\lambda_j)}{\lambda_i-\lambda_j} \\
    &=
    -\frac{1}{2} 
    \sum_{i,j=1}^n
    (O^t W O)_{i,j} q_iq_j
    \Delta^{\lambda_j}_{\lambda_i} B(\lambda_i,\lambda_i).
  \end{aligned}
  \end{equation*}
  All other solutions to the equation differ from this one by elements of
  \[
    \Span\left\{ E^{-}(\lambda_i,\lambda_i), i=1,2,\dots,n \right\}.
  \]
  \label{thm:comm}
\end{theorem}
\begin{proof}
  From \eqref{eq:spec1} and \eqref{eq:spec2} we can expand $W$ as
  \begin{align*}
    W &= \sum_{i=1}^n (O^t W O)_{i,i} \cdot q_i^2 E^{+}(\lambda_i,\lambda_i)\\
     &+ \sum_{i > j}^n (O^t W O)_{i,j} \cdot q_iq_j \cdot
     \left\{ 
       E^{+}(\lambda_i,\lambda_j) + E^{+}(\lambda_j,\lambda_i)
     \right\}.
  \end{align*}
  Applying \eqref{eq:H-} and \eqref{eq:F}, the theorem now follows.  Theorem~\ref{thm:span} shows that any other solution to the equation differs from this one as desired.
  The uniqueness of $M$ follows from the invertibility of $O,$ i.e.\,the linear independence of $\left\{ (p_{k-1}(\lambda_i))_k, i=1,\dots,n \right\},$ which is what appears in the first column of $E^{-}(\lambda_i,\lambda_i)$ for $i=1,\dots,n.$ 
\end{proof}

If we restrict to the case that $W$ is tridiagonal itself, this leads to a family of orthogonality rules which have quartic dependence on the orthogonal polynomials. 
\begin{corollary}
  \label{cor:orth}
  For a symmetric tridiagonal matrix $W,$
    \[
    W + 
    \sum_{i=1}^n (O^t W O)_{i,i} q_i^2 G(\lambda_i)
    =
    \sum_{i > j}^n (O^t W O)_{i,j} q_iq_j
    \frac{B(\lambda_i,\lambda_i) - B(\lambda_j,\lambda_j)}{\lambda_i-\lambda_j}.
  \]
\end{corollary}
\begin{proof}
  Suppose that $W$ is tridiagonal.  Then on the one hand, $M=M_1+M_2,$ given in the theorem, is one solution to
  \[
    [A,M] - W \in \TS.
  \]
  A second solution is given by $M=0.$  As both vanish in the first column, we have that $M_1+M_2 = 0.$ 
\end{proof}

We can also use Theorem~\ref{thm:comm} to make a connection to the minor polynomials.  Suppose we wished to solve for a matrix $W$ that is supported in a single entry
\begin{corollary}
  Let $W^{k,\ell}$ be the matrix that is $1$ in the $(k,\ell)$ and $(\ell,k)$ entries, for $k > \ell$, and $0$ elsewhere.  
  The matrix $M^{k,\ell}$ with $0$ in the first column so that $[A,M^{k,\ell}] = W^{k,\ell} + \TS$ is given by
  \[
    \begin{aligned}
    M
    &=
    -\frac{1}{2}\sum_{i,j=1}^n
    q_i^2
    q_j^2
    p_{k-1}(\lambda_i)
    p_{\ell-1}(\lambda_j)
    \Delta_{\lambda_j}^{\lambda_i} E^{-}(\lambda_i,\lambda_j), \\
    M_{u,r}
    &=
    -\frac{1}{2}\sum_{i}
    q_i^2
    p_{u-1}(\lambda_i)
    p_{k-1}(\lambda_i)p^{(\ell-1)}_{r-1}(\lambda_i),
  \end{aligned}
  \]
  for any $u > r.$
  When $k =\ell+1,$ this matrix is identically $0.$ We also define $M = 0$ if $k = \ell.$ 
  For other $k > \ell+1,$ the entry $M_{u,r}$ vanishes when $u-r \geq k-\ell,$ and it vanishes when $u+r \leq k+\ell.$
  The tridiagonal matrix $G^{k,\ell} = [A,M^{k,\ell}] - W^{k,\ell}$ is given by 
  \[
    {G}^{k,\ell}
    =
    -
    \frac{1}{2}\sum_{i,j=1}^n
    q_i^2
    q_j^2
    p_{k-1}(\lambda_i)
    p_{\ell-1}(\lambda_j)
    \Delta_{\lambda_j}^{\lambda_i} B(\lambda_i,\lambda_j)
  \]
  When $k = \ell+1,$ this matrix is equal to $W^{k,\ell},$ and when $k=\ell,$ this matrix is $0$ everywhere except in the $(k,k)$ entry where it is $1.$
  \label{cor:entry}
\end{corollary}
\begin{proof}
  We expand the $(u,r),$ entry of $M=M_1+M_2$ for $u\geq r$ from Theorem~\ref{thm:comm}, which gives
  \[
    M
    =
    -\frac{1}{2}\sum_{i,j}
    p_{k-1}(\lambda_i)
    p_{\ell-1}(\lambda_j)
    q_i^2 q_j^2
    \left( p_{u-1}(\lambda_i) + p_{u-1}(\lambda_j) \right)
    \left\{ 
      \frac{p_{r-1}(\lambda_i) - p_{r-1}(\lambda_j)}{\lambda_i-\lambda_j}
    \right\},
  \]
  where the $i=j$ terms are defined by continuity.
  In terms of the minor polynomials, this can be written as
  \[
    \begin{aligned}
    M=
    &-\frac{1}{2}\sum_{i}
    q_i^2
    p_{k-1}(\lambda_i) 
    p_{u-1}(\lambda_i)
    p^{(\ell-1)}_{r-1}(\lambda_i) \\
    &
    -\frac{1}{2}\sum_{j}
    q_j^2
    p_{\ell-1}(\lambda_j) 
    p_{u-1}(\lambda_j)
    p^{(k-1)}_{r-1}(\lambda_j).
    \end{aligned}
  \]
  As $k > \ell$ and $u > r,$ the second sum vanishes in all cases as $p_{\ell-1} p^{(k-1)}_{r-1}$ has degree at most $r-1 < u - 1.$
  To see the other vanishing conditions, if $u$ is too large, then $p_{u-1}$ is orthogonal to all polynomials of degree strictly smaller degree.  On the other hand, if $u+r$ is too small, then by the orthogonality of $p_{k-1}$ to lower degree polynomials, the first sum vanishes.  
\end{proof}

\begin{remark}
  It is also possible to express the matrix $G^{k,\ell}$ in terms on minor polynomials.
  \[
    G^{k,\ell}_{u,r}
    =
    \frac{1}{2}\sum_{i=1}^n
    q_i^2
    p_{k-1}(\lambda_i)
    \begin{cases}
    a_{r-1} p_{r-2}(\lambda_i)p_{r-1}^{(\ell-1)}(\lambda_i) - a_r p_{r-1}(\lambda_i) p_{r}^{(\ell-1)}(\lambda_i)&  \\
    +a_{r-1} p_{r-1}(\lambda_i)p_{r-2}^{(\ell-1)}(\lambda_i) - a_r p_{r}(\lambda_i) p_{r-1}^{(\ell-1)}(\lambda_i),  & \text{ if } u = r  < n, \\
    a_r (p_{r-1}(\lambda_i)p_{r-1}^{(\ell-1)}(\lambda_i) - p_r(\lambda_i)p_r^{(\ell-1)}(\lambda_i)), & \text{ if } u = r + 1 < n,\\
    0, & \text{ otherwise.}
  \end{cases}
\]
\label{rem:Gkl}
This characterization gives the result in Theorem \ref{thm:lanczos}
\end{remark}

This leads us directly to a characterization of all differentials $dM_t$ that produce tridiagonal models.

\section{The Tridiagonal Models}
\label{sec:models}

The first characterization of the tridiagonal models associated to is a direct consequence of the work in the previous section. The following theorem gives the martingale structure of the model. 

\begin{theorem}
\label{thm:martingales}
  The rotation differentials $dM_t$ of Dyson Brownian Motion $(\lambda_{i,t}, 1 \leq i \leq n)$ that produce \emph{nondegenerate} tridiagonal models are all of the form
  \[
    dM_t = \sum_{i=1}^n \left\{q_{i,t}^2 d\lambda_{i,t} F(\lambda_{i,t}) + c_i E^{-}(\lambda_{i,t},\lambda_{i,t})\right\} + dM^{(2)}_t.
  \]
  for some choice of coefficients $c_i$ and some finite variation process $dM^{(2)}_t.$  In particular, the martingale portion of $dA_t$ has the form
  \[
    \sum_{i=1}^n 
    \biggl\{
      -
      dZ_{i,t}
      \cdot q_{i,t}^2 \cdot G(\lambda_{i,t})
      +
      dX_{i,t}
      B(\lambda_{i,t},\lambda_{i,t})
    \biggr\}
  \]
  where $Z_{i,t}$ are the Brownian motions driving the Dyson Brownian motion and $X_{i,t}$ are some martingales. For the definitions of the matrices $F, E^{-}, G,$ and $B$ see the equations \textup{(\ref{eq:F}), (\ref{eq:E}), (\ref{eq:G})}, and \textup{(\ref{eq:B})} respectively
  \label{cor:trim}
\end{theorem}
There is a great deal of freedom in the general model. We specialize to two cases before writing down explicit finite variation processes, though this can be done for the general model. The two cases we study are the model where the spectral weights $q_i$ are frozen, and the second is where they evolve according to a finite variation process.

\subsection{The frozen spectral weight model}
\label{sec:tridiagonalfrozen}

Having the martingale portions of the SDE, we can then turn to solving for the finite variation portions of $dA_t.$  These second--order terms depend on the choices made in the martingale terms.  One possible choice is to make $dM_t$ have $0$ in the first column, which will lead to the frozen spectral weight model (where the spectral weights are constant in time).

\begin{theorem}
\label{thm:tridiagonalfrozen}
Let the $q_i$ be fixed then the tridiagonal model associated to the $\beta$-Dyson flow satisfies 
\begin{equation}
  \begin{aligned}
    dA_t &= \sum_{i=1}^n -\biggl\{\sqrt{\frac{2}{\beta}} dZ_{i,t} +  dt \cdot \biggl\{-\frac{V'(\lambda_{i,t})}{2} + \sum_{j \neq i}\frac{1}{\lambda_{i,t} - \lambda_{j,t}} \biggr\} \biggr\} \cdot q_i^2 \cdot G(\lambda_{i,t})
    + \sum_{k \geq \ell} dP_{k,\ell,t} \cdot G^{k,\ell} \\
    dP_t &= \sum_{i=1}^n -\frac{q_{i}^4}{2} \cdot [ -E^{+}(\lambda_{i,t},\lambda_{i,t}) + G(\lambda_{i,t}), F(\lambda_{i,t})] \cdot dt  .
  \end{aligned}
    \label{eq:frozen}
\end{equation}
Moreover if $\Lambda$ is a stationary process then $A$ will be stationary in the frozen spectral weight model. The definitions of $G, E^+, F$ and $G^{k,\ell}$ may be found in \textup{(\ref{eq:G}), (\ref{eq:E}), (\ref{eq:F})} and Remark \ref{rem:Gkl} respectively
\end{theorem}

\begin{corollary}
If we take $A_0$ to have the same distribution as the tridiagonal model in (\ref{eq:tridiagonal}) and $\Lambda$ satisfies the $\beta$-Dyson Brownian motion flow with $V(x)= 2x^2$ then $A_t$ will have the same distribution as (\ref{eq:tridiagonal}).
\end{corollary}

Recall that $dA_t$ satisfies
\[
  dA_t = dW_t + [A_t, dM_t] + \frac{1}{2}[2dW_t + [A_t,dM_t], dM_t]
\]
where $dW_t = \sum_{i=1}^n d\lambda_{i,t} q_i^2 E^{+}(\lambda_i,\lambda_i).$  
Define
\begin{equation}
  \label{eq:frozenM}
  \begin{aligned}
    dM^{(1)}_t &= \sum_{i=1}^n -\sqrt{\frac{2}{\beta}}dZ_{i,t} \cdot q_i^2 \cdot F(\lambda_{i,t}), \\
    dM^{(2)}_t &= \sum_{i=1}^n -dt \cdot q_i^2 \cdot F(\lambda_i)\cdot \biggl\{-\frac{V'(\lambda_{i,t})}{2} + \sum_{j \neq i}\frac{1}{\lambda_{i,t} - \lambda_{j,t}} \biggr\}, \\
    dM^{(3)}_t &= \sum_{k \geq \ell} -\frac{1}{2}[2dW_t+[A_t,dM^{(1)}_t], dM^{(1)}_t]_{k,\ell} \cdot M^{k,\ell},
  \end{aligned}
\end{equation}
and let $dM_t$ be the sum of these three components.  Here $M^{k,\ell}$ is the matrix defined in Corollary~\ref{cor:entry}.  Observe that the first column of $dM_t$ is $0.$  Recalling the definitions of $F$ and $M^{k,\ell},$ we have that
\[
  \begin{aligned}
  &[A_t, dM^{(1)}_t+ dM^{(2)}_t] \in \sum_{i=1}^n -d\lambda_{i,t} \cdot q_i^2 \cdot E^{+}(\lambda_i,\lambda_i) + \TS \\
  &[A_t, dM^{(3)}_t] \in -\frac{1}{2}[ 2dW_t + [A_t,dM^{(1)}_t],dM^{(1)}_t] + \TS \\ 
  &[A_t, dM^{(3)}_t] \in -\frac{1}{2}[ 2dW_t + [A_t,dM_t],dM_t] + \TS 
  \end{aligned}
\]
The equivalence of the last two terms comes from the fact that only products of two stochastic differentials survive.
Hence, this choice of $dM_t$ produces a tridiagonal diffusion, moreover, the produces the tridiagonal model in Theorem \ref{thm:tridiagonalfrozen}. To arrive at the expression we have for $dP_t,$ we have used the independence of the Brownian motions $dZ_{i,t}.$

\subsubsection*{Computation of $dP_t$}

This computation reduces to that of the commutators $[G(\lambda_i),F(\lambda_i)]$ and $[E^{+}(\lambda_i,\lambda_i),F(\lambda_i)]$ for an eigenvalue $\lambda_i.$  We further identify the dominant terms in this expansion, at least for a sufficiently small principal submatrix.
Specifically we show that
  \[
    \sum_{k \geq \ell} dP_{k,\ell,t} \cdot G^{k,\ell}
    =\sum_{i=1}^n \frac{q_i^2}{4} \partial_{xy} B(x,y) \vert_{x=y=\lambda_{i,t}}
    + dR_{t},
  \]
  where $dR_t$ is a lower order term, at least for sufficiently small windows of the upper left corner.  In the asymptotic regime considered in Section~\ref{sec:asymptotics}, the magnitude of the displayed terms in principal submatrix of size $M$ will be $M^2/n.$  In contrast, the $dR_t$ terms will be $M^4/n^2.$

We begin with $[G(\lambda_i), F(\lambda_i)].$  For notational simplicity, we suppress $\lambda_i$ in what follows.  All polynomials are evaluated at $\lambda_i.$
As $G$ is tridiagonal, we recall for convenience the commutator equation \eqref{eq:AM}
\begin{equation*}
[A,M]_{k,\ell} = a_{k-1} m_{k-1,\ell} + b_k m_{k,\ell} +a_k m_{k+1,\ell} - (a_{\ell-1} m_{k,\ell-1} + b_\ell m_{k,\ell}+ a_\ell m_{k,\ell+1}).
\end{equation*}
For $k > \ell +1,$ we get
\begin{align*}
  [G,F]_{k,\ell}
  =
  &-a_{k-1} (p_{k-2}p_{k-2}' - p_{k-1}p_{k-1}')
  p_{k-2}p'_{\ell-1} \\
  &-\biggl\{a_{k-1} p_{k-2}p_{k-1}' - a_k p_{k-1} p_{k}'
  +a_{k-1} p_{k-1}p_{k-2}' - a_k p_{k} p_{k-1}'\biggr\}
  p_{k-1}p'_{\ell-1} \\
  &-a_k(p_{k-1}p_{k-1}' - p_k p_k')
  p_{k}p'_{\ell-1} \\
  &+a_{\ell-1} (p_{\ell-2}p_{\ell-2}' - p_{\ell-1}p_{\ell-1}')
  p_{k-1}p'_{\ell-2} \\
  &+\biggl\{a_{\ell-1} p_{\ell-2}p_{\ell-1}' - a_\ell p_{\ell-1} p_{\ell}'
  +a_{\ell-1} p_{\ell-1}p_{\ell-2}' - a_\ell p_{\ell} p_{\ell-1}'\biggr\}
  p_{k-1}p'_{\ell-1} \\
  &+a_\ell(p_{\ell-1}p_{\ell-1}' - p_\ell p_\ell')
  p_{k-1}p'_{\ell}. 
\end{align*}
There is cancellation in this formula, which results in 
\begin{equation}
  \label{eq:GF}
  \begin{aligned}
  ~[G,F]_{k,\ell}
  = 
  \mathcal{H}_{k,\ell}
  :=
  &-p_{\ell-1}'a_{k-1}p_{k-2}'(p_{k-2}^2 + p_{k-1}^2) \\
  &+p_{\ell-1}'a_{k}p_{k}'(p_{k-1}^2 + p_{k}^2) \\
  &+p_{k-1}a_{\ell-1}p_{\ell-2}((p_{\ell-2}')^2 + (p_{\ell-1}')^2) \\
  &-p_{k-1}a_{\ell}p_{\ell}((p_{\ell-1}')^2 + (p_{\ell}')^2). 
\end{aligned}
\end{equation}
Here, the formula for $[G,F]$ holds for $k > \ell+1.$  We let $\mathcal{H}_{k,\ell}$ be given by this formula for all $k \geq \ell.$ 
From the symmetry of $G$ and the antisymmetry of $F,$ the commutator $[G,F]$ is symmetric, which allows its entries above the diagonal to be completed.  
Hence we extend the definition of $\mathcal{H}$ to be $0$ on the diagonal and symmetric.
On the tridiagonal, there are correction terms coming from the behavior of $F$ near the diagonal.  

The correction to \eqref{eq:GF} when $k = \ell+1$ stems from the $0$ diagonal of $F.$  This $0$-diagonal implies that \[
  G_{\ell+1,\ell} F_{\ell,\ell} - F_{\ell+1,\ell+1} G_{\ell+1,\ell} = 0.
\]
Hence, the commutator for $k=\ell+1$ is
\begin{equation}
  \label{eq:GFe1}
  \begin{aligned}
  ~[G,F]_{\ell+1,\ell}
  -\mathcal{H}_{\ell+1,\ell}
  &=
  -G_{\ell+1,\ell} p_{\ell-1} p_{\ell-1}' + p_{\ell}p_{\ell}' G_{\ell+1,\ell} \\
  &=
  -a_\ell( p_{\ell-1} p_{\ell-1}' - p_{\ell}p_{\ell}')^2. 
  \end{aligned}
\end{equation}
As for the diagonal, we get
\begin{equation}
  \label{eq:GFe2}
  \begin{aligned}
  ~[G,F]_{\ell,\ell}
  =&
  2G_{\ell,\ell+1} F_{\ell+1,\ell}
  -2G_{\ell-1,\ell} F_{\ell,\ell-1} \\
  =&-p_{\ell-1}'a_{\ell-1}p_{\ell-2}'(0 + 2p_{\ell-1}^2) \\
  &+p_{\ell-1}'a_{\ell}p_{\ell}'(0 + 2p_{\ell}^2) \\
  &+p_{\ell-1}a_{\ell-1}p_{\ell-2}(2(p_{\ell-2}')^2 + 0) \\
  &-p_{\ell-1}a_{\ell}p_{\ell}(2(p_{\ell-1}')^2 + 0).  \\
  =&\mathcal{H}_{\ell,\ell} \\
  &-p_{\ell-1}'a_{\ell-1}p_{\ell-2}'(-p_{\ell-2}^2 + p_{\ell-1}^2) \\
  &+p_{\ell-1}'a_{\ell}p_{\ell}'(-p_{\ell-1}^2 + p_{\ell}^2) \\
  &+p_{\ell-1}a_{\ell-1}p_{\ell-2}((p_{\ell-2}')^2 - (p_{\ell-1}')^2) \\
  &-p_{\ell-1}a_{\ell}p_{\ell}((p_{\ell-1}')^2 - (p_{\ell}')^2).  \\
  \end{aligned}
\end{equation}

Turning to $[E^{+},F],$ once again evaluated at an eigenvalue $\lambda_i$ which we suppress, for any $k \geq \ell$
\begin{align*}
  [E^{+},F]_{k,\ell}
  =&
  \sum_{m=1}^n E^{+}_{k,m} F_{m,\ell} - F_{k,m} E^{+}_{m,\ell} \\
  =&
  \sum_{m=\ell+1}^n -p_{k-1} p_{m-1} p_{m-1} p_{\ell-1}' 
  -\sum_{m=k+1}^n p_{k-1}' p_{m-1} p_{m-1} p_{\ell-1} \\
  +&\sum_{m=1}^{\ell-1} p_{k-1} p_{m-1} p_{m-1}' p_{\ell-1}
  +\sum_{m=1}^{k-1} p_{k-1} p_{m-1}' p_{m-1} p_{\ell-1}. \\
\end{align*}
Applying the Christoffel--Darboux identities to these sums, we arrive at
\begin{align*}
  [E^{+},F]_{k,\ell}
  =
  &-\{p_{k-1}p_{\ell-1}'+p_{k-1}'p_{\ell-1}\}q_i^{-2} \\
  &-p_{k-1}p_{\ell-1}'\left\{ - a_\ell p_\ell'p_{\ell-1}+a_\ell p_\ell p_{\ell-1}' \right\} \\
  &-p_{k-1}'p_{\ell-1}\left\{ - a_k p_k'p_{k-1}+ a_k p_k p_{k-1}' \right\} \\
  &+p_{k-1}p_{\ell-1}a_{\ell-1} \left\{ p_{\ell-1}''p_{\ell-2} - p_{\ell-1}p_{\ell-2}'' \right\} \\
  &+p_{k-1}p_{\ell-1}a_{k-1} \left\{ p_{k-1}''p_{k-2} - p_{k-1}p_{k-2}'' \right\}.
\end{align*}
We set $\mathcal{E}_{k,\ell}$ to be the sum of the final four lines from this equation.

\subsubsection*{Some simplifications of the FV terms}

There is some cancellation in the finite variation expressions which assists in estimating its true magnitude when passing to asymptotics. Much of the simplification is possible on account of the observation that for $k < \ell,$
\[
    \frac{1}{2}\sum_{i,j=1}^n
    q_i^2
    q_j^2
    p_{k-1}(\lambda_i)
    p_{\ell-1}(\lambda_j)
    \Delta_{\lambda_j}^{\lambda_i} B(\lambda_i,\lambda_j)
    =0.
\]
The left hand side of this expression gives $-G^{k,\ell}$ for $k \geq \ell.$  Hence if we extend the definition of $G^{k,\ell}$ to be $0$ for $k < \ell,$ we can express the finite variation terms as four sums
\begin{equation}
  \sum_{k \geq \ell} dP_{k,\ell,t} \cdot G^{k,\ell}
  = 
  \begin{aligned}
  &\sum_{i,k,\ell=1}^n -\frac{q_i^2}{2} \cdot (p_{k-1}(\lambda_i)p_{\ell-1}'(\lambda_i)+p_{k-1}(\lambda_i)'p_{\ell-1}(\lambda_i) )\cdot G^{k,\ell} \cdot dt  &\bigg\} =: dS_t \\
  +&\sum_{i,k,\ell=1}^n -\frac{q_i^4}{2} \cdot \mathcal{H}_{k,\ell}(\lambda_i)\cdot G^{k,\ell} \cdot dt \qquad &\bigg\} =: dR^1_t \\
  +&\sum_{i,k,\ell=1}^n -\frac{q_i^4}{2} \cdot ( [G(\lambda_i),F(\lambda_i)]_{k,\ell} - \mathcal{H}_{k,\ell}(\lambda_i))\cdot G^{k,\ell} \cdot dt \qquad &\bigg\} =: dR^2_t \\
  +&\sum_{i,k,\ell=1}^n \frac{q_i^4}{2} \cdot \mathcal{E}_{k,\ell}(\lambda_i)\cdot G^{k,\ell} \cdot dt.  &\bigg\} =: dR^3_t \\
\end{aligned}
  \label{eq:FVterms}
\end{equation}

Define for any eigenvalues $\lambda_i$ and $\lambda_u$
\[
  \mathfrak{D}(\lambda_u,\lambda_i)
  =\sum_{\ell=1}^n p_{\ell-1}(\lambda_u) p_{\ell-1}'(\lambda_i).
\]
Observe that for any $1 \leq r \leq n,$
\[
  \sum_{u=1}^n 
  \mathfrak{D}(\lambda_u,\lambda_i)
  p_{r-1}(\lambda_u) q_u^2 = p_{r-1}'(\lambda_i),
\]
and hence we have that for any polynomial $p$ of degree less than $n,$
\begin{equation}
  \label{eq:Did}
  \sum_{u=1}^n 
  \mathfrak{D}(\lambda_u,\lambda_i)
  p(\lambda_u) q_u^2 = p'(\lambda_i).
\end{equation}

\subsubsection*{Simplifying $dS_t:$}

We begin this case by observing
that since $G^{k,\ell} = 0$ for $k < \ell,$
\[
\sum_{i,k,\ell=1}^n -\frac{q_i^2}{2} \cdot (p_{k-1}(\lambda_i)p_{\ell-1}'(\lambda_i))\cdot G^{k,\ell} = 0,
\]
from the orthogonality of $p_{k-1}$ to lower degree polynomials.  Hence we only need consider the other term.  Expanding the definition of $G^{k,\ell},$
\begin{align*}
  dS_t
  &=
  \sum_{i,k,\ell=1}^n -\frac{q_i^2}{2} \cdot (p_{k-1}(\lambda_i)'p_{\ell-1}(\lambda_i) )\cdot G^{k,\ell} \cdot dt \\
  &=
  \sum_{i,k,\ell,u,j=1}^n 
  \frac{q_i^2q_u^2q_j^2}{4} \cdot 
  p_{k-1}(\lambda_i)'p_{\ell-1}(\lambda_i) \cdot 
  p_{k-1}(\lambda_u)
  p_{\ell-1}(\lambda_j)
  \Delta_{\lambda_j}^{\lambda_u} B(\lambda_u,\lambda_j)
  \cdot dt \\
  &=
  \sum_{i,k,u=1}^n 
  \frac{q_i^2q_u^2}{4} \cdot 
  p_{k-1}(\lambda_i)' \cdot 
  p_{k-1}(\lambda_u)
  \Delta_{\lambda_i}^{\lambda_u} B(\lambda_u,\lambda_i)\cdot dt \\
  &=
  \sum_{i,u=1}^n 
  \frac{q_i^2q_u^2}{4} \cdot 
  \mathfrak{D}(\lambda_u,\lambda_i)
  \Delta_{\lambda_i}^{\lambda_u} B(\lambda_u,\lambda_i)\cdot dt. 
\end{align*}
In the third equality, we have used the orthogonality of the polynomials in $\ell.$ In the fourth, we have used the definition of $\mathfrak{D}.$ Hence, we conclude
\begin{lemma}
  For $k,\ell$ satisfying $2\max(k,\ell) \leq n+1,$ we have
  \[
    dQ_{k,\ell,t} = 
    \sum_{i=1}^n 
    \frac{q_i^2}{4} \partial_{xy}B_{k,\ell}(x,y) \vert_{x=y=\lambda_{i,t}}
    .
  \]
  \label{lem:R3}
\end{lemma}
\begin{proof}
  For $k,\ell$ as stated,
  \(
    \Delta_{y}^{x}B_{k,\ell}(x,y)
  \)
  are polynomials in $x$ of degree at most $n-1.$
  Applying \eqref{eq:Did} to $dS_t,$ it follows that 
  \[
    dQ_{k,\ell,t} = 
    \sum_{i=1}^n 
    \frac{q_i^2}{4} (\partial_{y} \Delta_{\lambda_i}^{y}B_{k,\ell}(y,\lambda_i))\vert_{y=\lambda_i}
  \]
  Expanding,
  \[
    \partial_{y} \Delta_{\lambda_i}^{y}B(y,\lambda_i)
    =
    \partial_y
    \frac{B(y,y) - B(y,\lambda_i)}{y-\lambda_i}.
  \]
  Hence, on setting $y=\lambda_i,$ we get
  \[
    (\partial_{y} \Delta_{\lambda_i}^{y}B(y,\lambda_i))\vert_{y=\lambda_i}
    =(\partial_{xy} + \frac{1}{2}\partial_{yy})B(x,y) \vert_{x=y=\lambda_i}.
  \]
  By orthogonality, and recalling \eqref{eq:B}, we have
  \[
    \sum_{i=1}^n 
    \frac{q_i^2}{4}
    (\partial_{yy})B(x,y) \vert_{x=y=\lambda_i} = 0,
  \]
  which completes the proof.
\end{proof}

\subsubsection*{Simplifications for quadratic constraining potential}
For $V'(t) = ct,$ some additional simplification is possible.  Specifically, we may take advantage of the identity
\begin{equation}
  \sum_{i=1}^n q_i^2 \lambda_i G(\lambda_i) =  -A
  \label{eq:AG}
\end{equation}
We prove this like follows.  By \eqref{eq:F}, for an eigenvalue $\lambda$ of $A,$
\[
    [A,F(\lambda)] = E^{+}(\lambda,\lambda) + G(\lambda).
\]
Recall that an entry of $F_{k,\ell}(\lambda)=p_{k-1}(\lambda)p_{\ell-1}'(\lambda)$ for $k > \ell.$ Hence, by orthogonality of $p_{k-1}$ to lower degree polynomials.
\[
  \sum_{i=1}^n q_i^2 \lambda_i F(\lambda_i) = 0.
\]
On the other hand, recalling \eqref{eq:spec2},
\[
  \sum_{i=1}^n E^{+}(\lambda_i,\lambda_i)q_i^2 \lambda_i = A,
\]
which completes the proof.

\subsection{Smooth spectral weight model}
\label{sec:tridiagonalfv}

A relatively tame perturbation of the frozen--spectral weight model is one in which the spectral weights are finite variation processes, which may or may not depend on the eigenvalues.  
\begin{theorem}
\label{thm:tridiagonalFV}
Suppose that $dR_{i,t},$ for $1 \leq i \leq n$ are any finite variation processes.  Define
\[
  dM^R_t = 
  dM^{F}_t
  + \sum_{i=1}^n dR_{i,t} q_{i,t} E^{-}(\lambda_{i,t},\lambda_{i,t}),
\]
where $dM^{F}_t$ is the rotation differential from \eqref{eq:frozenM}.  The associated tridiagonal model is given by 
\begin{equation}
  \begin{aligned}
    dA_t &= dA^{F}_t + \sum_{i=1}^n dR_{i,t} q_i B(\lambda_{i,t},\lambda_{i,t}).
  \end{aligned}
  \label{eq:smooth}
\end{equation}
\end{theorem}

\begin{proof}
The eigenvector matrix $O_t$ of $A_t$ evolves (by Theorem~\ref{thm:tri}) as
\begin{equation*}
  dO_t^t = O_t^t (dM_t +\frac{1}{2} dM_t^2).
\end{equation*}
When $dM_t = dM^R_t,$ the first column of $dM_t^2$ vanishes, as $dM^F_t$ is identically $0$ in this column and $dM^R_t-dM^F_t$ is finite variation.  Hence, the first row of $dO_t$ evolves according to
\begin{align*}
  dq_{i,t} = dO_{1,i,t} = 
  \sum_{\ell=1}^n O_{\ell,i} dM^R_{\ell,1,t} 
  &=
  \sum_{\ell=2}^n q_{i,t} p_{\ell-1}(\lambda_{i,t}) 
  \biggl\{ 
    \sum_{j=1}^n dR_{j,t} q_{j,t} p_{\ell-1}(\lambda_{j,t})
  \biggr\} 
  \\
  &=
  -q_i \sum_{j=1}^n dR_{j,t} q_{j,t} 
  +
  \sum_{j=1}^n 
  \sum_{\ell=1}^n q_{i,t} p_{\ell-1}(\lambda_{i,t}) 
    dR_{j,t} q_{j,t} p_{\ell-1}(\lambda_{j,t}) \\
  &=
  -q_{i,t} \sum_{j=1}^n dR_{j,t} q_{j,t}
  + dR_{i,t}.
\end{align*}
In other words, the spectral weights evolve exactly according to the projection of $( dR_{i,t} )_{i=1}^n$ in the direction orthogonal to $(q_{i,t})_{i=1}^n.$  Note that
\[
  \sum_{i=1}^n q_{i,t}^2 E^{-}(\lambda_{i,t},\lambda_{i,t}) = 0,
\]
and hence one may as well assume that $dR_{i,t}$ already had this orthogonality.
\end{proof}

\subsection{The half finite variation model}

The frozen spectral weight and smooth spectral weight models featured in some sense the tamest choices for the evolution of the spectral weights.  However, it is also possible to choose evolutions of the spectral weights that lead to substantially different tridiagonal evolutions.  For example, one might hope to choose a spectral weight evolution that cancels some of the rough parts of the tridiagonal.  Here, we show it is possible to choose these weights such that the first half of the tridiagonal model will be finite variation.

In effect, it suffices to show the following lemma.
\begin{lemma}
There is a choice of coefficients $\{c^j_i\}$ so that
\[
  F(\lambda_i) = \sum_{j=1}^n c^j_i E^{-}(\lambda_j,\lambda_j) + T,
\]
for some $T$ that vanishes for entries $(k,\ell)$ with $k+\ell \leq n.$
  \label{lem:Fsolve}
\end{lemma}
\begin{proof}
  Suppose that $p$ is a polynomial of degree at most $n-1.$  Then
  \[
    \sum_{i=1}^n \Delta_{\lambda_i}^{\lambda_j} p(\lambda_i) p_{n-1}(\lambda_i) q_i^2 = 0
  \]
  by orthogonality.
  It follows that on rearranging the sum
  \[
    \sum_{i \neq j} \Delta_{\lambda_i}^{\lambda_j} p(\lambda_i) p_{n-1}(\lambda_i) q_i^2
    =
    p'(\lambda_j) p_{n-1}(\lambda_j) q_j^2.
  \]
  In particular, we have for $k+\ell \leq n,$ when $E^{-}_{k,\ell}(x,x)$ has degree at most $n-1,$ 
  \[
    F_{k,\ell}(\lambda_i) 
    = -\frac{1}{2}\partial_{\lambda_i} E^{-}(\lambda_i,\lambda_i)
    = \frac{-1}{2p_{n-1}(\lambda_i)q_i^2}\sum_{j \neq i} \Delta_{\lambda_j}^{\lambda_i}  E^{-}_{k,\ell}(\lambda_j,\lambda_j) p_{n-1}(\lambda_j) q_j^2.
  \]
  Hence choosing $c^j_i$ to be these coefficients, the lemma follows.
\end{proof}

\begin{corollary}
  There is a tridiagonal model $A_t$ for Dyson Brownian motion so that its entries $(k,\ell)$ with $k+\ell \leq n$ are finite variation.
\end{corollary}

\section{Asymptotics}
\label{sec:asymptotics}

In this section we will prove the large--$n$ asymptotics given in Theorem \ref{thm:asymptotics} of the bounded--order principal submatrices of the frozen spectral weight model with external potential $V(x)=2x^2$ run from the Dumitriu--Edelman distribution.   
Let $\MP{A}$ be this tridiagonal matrix process.
Let $M$ be a fixed natural number, which will be the order of the principal submatrix that we consider.  Recall that the model is given by \eqref{eq:frozen}, which after applying \eqref{eq:AG} is given by
\begin{equation}
  \begin{aligned}
    dA_t &= -2A_t\,dt + \sum_{i=1}^n -dt \cdot \biggl\{\sum_{j \neq i}\frac{1}{\lambda_{i,t} - \lambda_{j,t}} \biggr\} \cdot q_i^2 \cdot G(\lambda_{i,t}) + dA_t' \\
    dA_t' &= \sum_{i=1}^n -\sqrt{\frac{2}{\beta}} dZ_{i,t}\cdot q_i^2 \cdot G(\lambda_{i,t})
    + \sum_{k \geq \ell} dP_{k,\ell,t} \cdot G^{k,\ell} \\
    dP_t &= \sum_{i=1}^n -\frac{q_i^4}{2} \cdot [ -E^{+}(\lambda_{i,t},\lambda_{i,t}) + G(\lambda_{i,t}), F(\lambda_{i,t})] \cdot dt
  \end{aligned}
  \label{eq:frozen2}
\end{equation}
To show the convergence statement we will show that the contribution from $A'$ will be negligible and the contributions from $dA_t-dA'_t$ will be derived using an observation about the orthogonal polynomials being approximately Chebyshev and convergence of Stieltjes transforms.

This section is organized in the following way: we first show the observation about the polynomials being approximately Chebyshev holds in a precise way. Second we show that the contribution of we prove convergence of the appropriate Stieltjes transforms. Third we show the contributions from the $dA'_t$ are negligible. Fourth we give the asypmtotics for the main terms in $dA_t- dA'_t$. Fifth we show the cancellation of the leading order terms. Lastly we collect everything to get the appropriate convergence statement.

\subsection{Convergence to Chebyshev polynomials}

The asymptotics are based off the observation that the orthogonal polynomials in this upper corner are nearly Chebyshev polynomials.
Define the semi--infinite tridiagonal matrix
\[
  D
  =
  \begin{bmatrix}
    0 & \frac{1}{2} &   &  &  & \\
    \frac{1}{2} & 0 & \frac{1}{2} &  &  & \\
      & \frac{1}{2} & 0 & \frac{1}{2}&  & \\
      &   & \frac{1}{2} &  &  &  & \\
      &   &   &  &  \ddots &  & \\
  \end{bmatrix}.
\]
The orthogonal polynomials with this Jacobi matrix are the Chebyshev polynomials of the second kind, $\left\{ U_k \right\}_{k=0}^{\infty}.$

Define stopping times, for any $R > 0,$
\[
  \tau_R = \inf\left\{ t \geq 0 ~:~ \max_{1 \leq i,j \leq M} |\MP{A}{t}[i,j]-\sqrt{n}D_{i,j}| > R \right\}.
\]
For time before $\tau_R,$ we have quantitative control on the proximity of $\MP{A}$ to $D,$ and this implies the closeness of the scaled orthogonal polynomials associated to $\MP{A}$ to the Chebyshev polynomials of the second kind.
\begin{lemma}
  Uniformly in $0 \leq k \leq M,$ $0 \leq j < k,$ $t \leq \tau_R,$ and locally uniformly in $x\in\C,$
  \begin{align*}
    &\lim_{n\to\infty} p_{k,t}(\sqrt{n} x) = U_k(x), \\
    &\lim_{n\to\infty} a_{j+1,t}p^{(j)}_{k,t}(\sqrt{n} x) = U_{k-j-1}(x), \\
    &\lim_{n\to\infty} \sqrt{n} p_{k,t}'(\sqrt{n} x) = U_k'(x).
  \end{align*}
  \label{lem:cheb}
\end{lemma}
\begin{proof}
  The three-term recurrence for $q_{k,t}(x) = p_{k,t}(\sqrt{n}x)$ is given by
  \begin{align*}
    x q_{k,t}(x) 
    =
    x p_{k,t}(\sqrt{n} x)
    &=
    \frac{a_{k+1,t}}{\sqrt{n}} p_{k+1,t}(\sqrt{n} x)
    +\frac{b_{k+1,t}}{\sqrt{n}} p_{k,t}(\sqrt{n} x)
    +\frac{a_{k,t}}{\sqrt{n}} p_{k-1,t}(\sqrt{n} x) \\
    &=
    \frac{a_{k+1,t}}{\sqrt{n}} q_{k+1,t}(x)
    +\frac{b_{k+1,t}}{\sqrt{n}} q_{k,t}(x)
    +\frac{a_{k,t}}{\sqrt{n}} q_{k-1,t}(x).
  \end{align*}
  From the uniform convergence of $n^{-1/2} A_t,$ the $3$-term recurrence converges to that of the Chebyshev polynomials.  Hence by induction on $k,$ the result follows. 
\end{proof}

We will need to find a uniform left--tail bound for these stopping times in $n$ which improves as $R \to \infty.$ To do so, define stopping times, for $K >0,$ 
\[
  \sigma_K = \inf\left\{ t \geq 0 ~:~ \max_{1 \leq i \leq N} |\lambda_{i,t}| > K\sqrt{n} \right\}.
\]
\begin{lemma}
  There is a $K >0$ so that for all $T >0,$
  \[
    \lim_{n \to \infty} \Pr\left[ \sigma_K \leq T \right] = 0.
  \]
  \label{lem:sigmaK}
\end{lemma}
\begin{proof}
  See \cite[Theorem 5.1]{Unterberger}.
\end{proof}

\subsection{Stieltjes transforms and their limits}

The Stieltjes transform of a measure $\mu$ on the real line is given by 
\[
s^\mu(z) =\int_\rr \frac{1}{x-z}d\mu.
\]
We can also define the Stieltjes transform of a matrix by taking the measure to be the associated spectral measure. We make use of both the unweighted and weighted spectral measures and so define the Stieltjes transforms of $A_t$ by
\[
  s_t(z) = \frac{1}{n}\sum_{i=1}^n \frac{1}{\lambda_{i,t}/\sqrt{n}-z}
  \quad
  \text{and}
  \quad
  s_t^A(z) = \sum_{i=1}^n \frac{q_i^2}{\lambda_{i,t}/\sqrt{n}-z},
\]
for $z$ in the complement of the spectrum with respect to the complex plane.
Recall that $\mathfrak{s}$ denotes the semicircle density on $[-1,1],$
\[
  \mathfrak{s}(x) = \frac{2}{\pi}\sqrt{1-x^2}\,\one[|x| \leq 1],
\]
and that we define for $z \in \C \setminus [-1,1]$
\[
  \SCS(z) =
  \int_{-1}^1 \frac{\mathfrak{s}(x)\,dx}{x-z}
  =
  2(-z+\sqrt{z^2-1})
  .
\]
\begin{proposition}
  \label{prop:stieltjes}
  Let $K,T > 0$ and let $F \subset \C \setminus [-K,K]$ be compact.  Then 
  \[
    \max_{0 \leq t \leq \sigma_K \wedge T} \max_{z \in F} |s_t(z) - \SCS(z)| \to 0
  \]
  pointwise. This is essentially a version of uniform convergence for $0 \le t \le T$ of $\mu_{n,t} = \frac{1}{n} \sum_{i=1}^n \delta_{\lambda_i/\sqrt{n}}$ weakly to $\mathfrak{s}$.
\end{proposition}

\begin{proof}
The first statement is a consequence of Theorem 1 in Rogers and Shi \cite{RogersShi}. To get that this implies the uniform weak convergence first observe that the measures are compactly supported and so $\int x^k d\mu_{n,t} \to \int x^k d\mathfrak{s}(x)$ by contour integration, which then implies the weak convergence.
\end{proof}

Let $Q_t$ be a standard Brownian motion defined for all positive and negative times, and set
\[
  \SQ(z) = 
  s^{\mathfrak{s}}(z)c_\beta \int_{-1}^{1}1\ dQ_x+ c_\beta \int_{-1}^1 \frac{\sqrt{\mathfrak{s}(x)}dQ_x}{x-z},
\]
with $c_\beta = \sqrt{\beta}.$
\begin{proposition}
  \label{prop:stieltjes2}
  Let $K,T > 0$ and let $F \subset \C \setminus [-K,K]$ be compact.  Then there is a probability space so that
  \[
    \max_{0 \leq t \leq \sigma_K \wedge T} \max_{z \in F} |\sqrt{n}(s_t^A(z)-s_t(z)) - \SQ(z)| \Prto 0.
  \]
\end{proposition}
\begin{proof}
We make use of the martingale central limit theorem to show convergence directly. Recall that the the vector $(q_1^2,...,q_n^2)$ has Dirichlet$(\frac{\beta}{2},...,\frac{\beta}{2})$ distribution. We couple the $q_i$'s to an independent family of i.i.d. random variables $Y_1,Y_2,...$ in the following way. Let $Y_i \sim$ Gamma$(\frac{\beta}{2},\theta)$ then for $V_n = \sum_{i=1}^n Y_i$ we define 
\[
\left( \frac{Y_1}{V_n},...,\frac{Y_n}{V_n}\right) = (q_1^2,...,q_n^2).
\]
Using this we can use the decomposition  $\sqrt {n} (s_t^A(z)-s_t(z)) = X_n^{(t)}(z)+M_n^{(t)}(1,z)$ where 
 \[
 X^{(t)}_n(z) = \sqrt n \sum_{i=1}^{n} \left( \frac{Y_i}{\theta\frac{\beta}{2}n} - q_i^2\right) \frac{1}{\lambda_{i,t}/\sqrt n-z} = \frac{V_n- \theta\frac{\beta}{2}n}{\theta \frac{\beta}{2}\sqrt n} s_t^A. 
\]
and $M_n^{(t)}$ is a family of continuous time martingales defined by 
\[
M_n^{(t)}(s,z) = \sqrt{n} \sum_{i=1}^{\lfloor sn\rfloor} \left( \frac{1}{n} - \frac{Y_i}{\theta \frac{\beta}{2} n}\right) \frac{1}{\lambda_{i,t}/\sqrt n-z}.
\]

To show convergence we begin by defining
\[
S_n(s) =\sqrt{\frac{2}{\beta}}\sqrt{n} \sum_{k=1}^{\lfloor sn\rfloor} \left(\frac{Y_i}{\theta \frac{\beta}{2} n}- \frac{1}{n}\right).
\]
A strong approximation theorem by Koml\'os, Major, and Tusn\'ady \cite{KMT} gives us that there exists a Brownian motion $B(s)$ so that  $\sup_{s\in[0,T] }|S_n(s) - B(s)| \Prto 0.$ Let $Q(x) = B((x+1)/2)$. Notice that $X_n^{(t)}$ is bounded an so $|s_t(z)-s_t^A(z)|\to 0$ for all $t$. In particular we get that $\sup_{0\le t \le T} \sup_{z\in F}| s_t^A(z)- s^{\mathfrak{s}}(z)|\to 0$ in probability. From this we get 
\begin{equation}
 \max_{0 \leq t \leq \sigma_K \wedge T} \max_{z \in F}\left| X^{(t)}_n(z) - s^{\mathfrak{s}}(z)c_\beta \int_{-1}^{1}1\ dQ_x \right| \to 0 \quad \text{ in probability.}
 \label{eq:Xlimit}
\end{equation}
We now turn to the martingale term. In order to prove convergence we will separate the $\lambda_{i,t}$ from the main term. Let $\gamma_i$ be the $i/n$th quantile of $\mathfrak{s}$. That is $\gamma_i$ satisfies
\[
\int_{-1}^{\gamma_i} \mathfrak{s}(x)dx = \frac{i}{n}.
\]
Now observe that $M_n^{(t)}(s)$ may be rewritten as
\[
M_n^{(t)}(s) = \sqrt{n} \sum_{i=1}^{\lfloor sn\rfloor} \left( \frac{1}{n} - \frac{Y_i}{\theta \frac{\beta}{2} n}\right)  \frac{1}{\gamma_i-z}
+\sqrt{n} \sum_{i=1}^{\lfloor sn\rfloor} \left( \frac{1}{n} - \frac{Y_i}{\theta \frac{\beta}{2} n}\right)\left(\frac{1}{\lambda_{i,t}/\sqrt n-z}-  \frac{1}{\gamma_i-z}\right).
\]
The second moment of the second term, conditional on the eigenvalue evolution, then may be bounded by 
\[
n  \max_{1\le i \le n}\left\{ \left(\frac{1}{\lambda_{i,t}/\sqrt n-z}-  \frac{1}{\gamma_i-z}\right)^2 \right\} \sum_{i=1}^{\lfloor sn \rfloor} \Exp\left( \frac{1}{n} - \frac{Y_i}{\theta \frac{\beta}{2} n}\right)^2.
\]
If we can show that the maximum converges to $0$ in probability that will show that the second term converges to 0 in probability. Recall that $\sup_{0\le t \le T} \max_{1\le i \le n}|\frac{1}{\lambda_{i,t}/\sqrt n-z}-  \frac{1}{\gamma_i-z}| \to 0$ in probability. This follows from convergence of the quantile functions which can be obtained from the weak convergence in Proposition \ref{prop:stieltjes}. 
We now show the first sum appearing in $M_n^{(t)}$ converges to a stochastic integral.
We make use of the semicircle convergence again to get that 
\[
  \frac{1}{n}
\sum_{i=1}^{\lfloor sn\rfloor}  \frac{1}{\gamma_i-z} \to \int_{-1}^{2s -1} \frac{ \mathfrak{s}(x)}{x-z} dx.
\]
This convergence is uniform in $s$.  By theorem 4.6 in Kurtz and Protter \cite{KP} we get that 
\[ 
\max_{0 \leq t \leq \sigma_K \wedge T} \max_{z \in F} \left|M_n^{(t)} - c_\beta \int_{-1}^1 \frac{\sqrt{\mathfrak{s}(x)}dQ_x}{x-z}\right| \Prto 0.
\]
This together with (\ref{eq:Xlimit}) completes the proof.
\end{proof}

Having shown that the Stieltjes transforms converge, we can show that stopping times $\tau_R$ advance to infinity for large $n$ and large $R.$
\begin{proposition}
  Let $T > 0,$ then 
  \[
    \limsup_{R \to \infty} \limsup_{n \to \infty} \Pr\left[ \tau_R \leq T \right] = 0.
  \]
  \label{prop:tightness}
\end{proposition}
\begin{proof}
  Let $K > 1$ be a constant as in Lemma~\ref{lem:sigmaK}.  Let $\gamma$ be a smooth contour winding once in the positive orientation around $[-K,K]$ in the complex plane. 
  Define the $k$-th moment
  \[
    c_{k,t} 
    = \sum_{i=1}^n q_i^2 (\lambda_{i,t}/n)^2 
    = \frac{1}{2\pi i} \oint_\gamma s^A_t(z) z^k\,dz.
  \]
  From \cite[(2.2.6)]{Szego}, it is possible to express for all $1 \leq k \leq n-1$
  \[
    p_k(\sqrt{n} x) = \frac{1}{\sqrt{D_kD_{k-1}}}
    \left|
    \begin{matrix}
      c_0 & c_1 & c_2 & \dots & c_k          \\
      c_1 & c_2 & c_3 & \dots & c_{k+1}      \\
      \dots & \dots & \dots & \dots & \dots  \\
      c_{k-1} & c_k & c_{k+1} & \dots & c_{2k-1} \\
      1 & x & x^2 & \dots & x^{k} \\
    \end{matrix}
    \right|,
  \]
  where $D_k$ is the Hankel determinant $\det[c_{i+j}]_{i,j=0,\dots, k}.$
  By Proposition~\ref{prop:stieltjes}, each coefficient of $p_k(\sqrt{n}x)$ converges with rate $1/\sqrt{n}$ to the corresponding coefficient of $U_k(x),$ the orthogonal polynomials associated to the semicircle, uniformly for time up to $T \wedge \sigma_K.$  That is to say, for any natural numbers $i,j,$
  \[
    \max_{0 \leq t \leq T \wedge \sigma_K} |\MP{A}{t}[i,j]/\sqrt{n} - D_{i,j}|/\sqrt{n} 
  \]
is tight as a family of variables in $n.$  Hence 
\[
    \limsup_{R \to \infty} \limsup_{n \to \infty} \Pr\left[ \tau_R \leq T \right] 
    \leq 
    \limsup_{n \to \infty}
    \Pr\left[ \sigma_K \leq T \right],
\]
which is equal to $0$ by Lemma~\ref{lem:sigmaK}.
\end{proof}

\subsection{Negligible terms}
Recall the decomposition given in (\ref{eq:frozen2}). In this section we show that the $A'_t$ terms will be negligible. Recall
\begin{equation*}
  \begin{aligned}
    dA_t' &= \sum_{i=1}^n -\sqrt{\frac{2}{\beta}} dZ_{i,t}\cdot q_i^2 \cdot G(\lambda_i)
    + \sum_{k \geq \ell} dP_{k,\ell,t} \cdot G^{k,\ell} \\
    dP_t &= \sum_{i=1}^n -\frac{q_i^4}{2} \cdot [ -E^{+}(\lambda_i,\lambda_i) + G(\lambda_i), F(\lambda_i)] \cdot dt
  \end{aligned}
\end{equation*}
Here we will also need that the spectral weights are relatively flat.  
\begin{lemma}
  There is a constant $C_\beta > 0$ so that with probability going to $1$
  \[
    \max_{1 \leq i \leq n} q_i^2 \leq \frac{C_\beta \log n}{n},
  \]
  as $n \to \infty.$
  \label{lem:qibound}
\end{lemma}

\begin{proposition}
  For any $T,R,K > 0$ as $n\to\infty,$
  \[
    \max_{1\leq i,j \leq M}
    \biggl|
    \int_0^{T \wedge \tau_R \wedge \sigma_K}
    dA_{i,j,t}'
    \biggr|
    \Prto
    0.
  \]
  \label{prop:dAt}
\end{proposition}
\begin{proof}
  Let $\zeta$ be the stopping time $T \wedge \tau_R \wedge \sigma_K.$
  We consider the local martingale and finite variation portions separately.  Recall that entries of $G$ are given by
  \begin{equation}
    G_{k,\ell}(\lambda)
  =\begin{cases}
    a_{\ell-1} p_{\ell-2}(\lambda)p_{\ell-1}'(\lambda) - a_\ell p_{\ell-1}(\lambda) p_{\ell}'(\lambda)&  \\
    +a_{\ell-1} p_{\ell-1}(\lambda)p_{\ell-2}'(\lambda) - a_\ell p_{\ell}(\lambda) p_{\ell-1}'(\lambda),  & \text{ if } k = \ell , \\
    a_\ell (p_{\ell-1}(\lambda)p_{\ell-1}'(\lambda) - p_\ell(\lambda)p_\ell'(\lambda)), & \text{ if } k = \ell + 1 ,\\
    0, & \text{ otherwise.}
  \end{cases}
  \label{eq:GG}
\end{equation}
Hence in light of Lemma~\ref{lem:cheb}, we have that
\[
  G^* = \max_{1 \leq i \leq n} \max_{0 \leq t \leq \zeta} \max_{1\leq k,\ell \leq M}
  |G_{k,\ell}(\lambda_{i,t})|
\]
is tight as a family of random variables in $n.$ Employing this bound, the quadratic variation of the local martingale is dominated by 
\[
  \biggl\langle
  \sum_{i=1}^n -\sqrt{\frac{2}{\beta}} dZ_{i,t}\cdot q_i^2 \cdot G_{k,\ell}(\lambda_i)
\biggr\rangle
\leq 
\sum_{i=1}^n \frac{2}{\beta}q_i^4 (G^*)^2\,dt.
\]
By Lemma~\ref{lem:qibound}, it follows that
\[
    \max_{1\leq k,\ell \leq M}
    \biggl|
    \int_0^{\zeta}
    \sum_{i=1}^n dZ_{i,t}\cdot q_i^2 \cdot G_{k,\ell}(\lambda_i)
    \biggr|
    \Prto
    0.
\]

As for the finite variation terms, we consider each of the four classes of terms broken out in \eqref{eq:FVterms}.  The estimates for $dS_t,$ $dR^1_t, dR^2_t,$ and $dR^3_t$ are similar
and so we explain the estimates for $dR^1_t$ and do not discuss the others.

Recall that
\(
dR^1_t=\sum_{i,k,\ell=1}^n -\frac{q_i^4}{2} \mathcal{H}_{k,\ell} \cdot G^{k,\ell} \cdot dt,
\)
where $\mathcal{H}_{k,\ell}$ is given in \eqref{eq:GF} and $G^{k,\ell}$ is given in Corollary~\ref{cor:entry}.  The entries $G^{k,\ell}_{u,r}$ vanish when $u+r < k+\ell.$  In particular, this implies that for considering $dR^{1}_{u,r,t}$ for $1 \leq u,r \leq M$ it suffices to estimate $k,\ell$ with $1 \leq k,\ell \leq 2M.$  Using Remark~\ref{rem:Gkl} and Lemma~\ref{lem:cheb}, we have that
\[
  G^{**}=
  \max_{0 \leq t \leq \zeta} 
  \max_{1 \leq k,\ell \leq 2M} 
  \max_{1 \leq u,r \leq M} 
  |G^{k,\ell}_{u,r,t}|
\]
is tight as a family of random variables in $n$, where we have used that $\sum_{i=1}^n q_i^2=1.$ In a similar way,
\[
  H^*
  =
  \max_{1 \leq i \leq n}
  \max_{0 \leq t \leq \zeta} 
  \max_{1 \leq k,\ell \leq 2M} 
  |\mathcal{H}_{k,\ell,t}(\lambda_{i,t})|\sqrt{n}
\]
is tight.  Hence, 
\[
  \biggl| \int_0^\zeta dR^1_t \biggr| \leq 
  \zeta \cdot 
  (\max_{1 \leq i \leq n} q_i^4) \cdot
  \sqrt{n} \cdot H^* \cdot G^{**} \Prto 0,
\]
by Lemma~\ref{lem:qibound}.


\end{proof}

\subsection{Principal terms}
The remaining terms will produce a nonvanishing contribution to the limit in the regime we consider.  Recalling \eqref{eq:frozen2}, these remaining terms are given by
\[
  dA_t - dA_t' = -2A_t\,dt + \sum_{i=1}^n -dt \cdot \biggl\{\sum_{j \neq i}\frac{1}{\lambda_i - \lambda_j} \biggr\} \cdot q_i^2 \cdot G(\lambda_i).
\]
There are large cancellations between the $2A_t$ term and the sum.  Both are order $\sqrt{n},$ but their leading order behavior cancels.  

To identify an exact leading term, we make a connection to the minor polynomials.
Recall from \eqref{eq:G} that $G(\lambda) = -\partial_\lambda B(x,\lambda) \vert_{x=\lambda}.$ By orthogonality, we have that 
\[
  \sum_{i=1}^n q_i^2\biggl\{\partial_{\lambda \lambda} B(\lambda_i,\lambda)\vert_{\lambda=\lambda_i}  \biggr\}=0,
\]
on account of the orthogonality of $p_{k}$ to lower degree polynomials.  Hence, we have
\[
  \sum_{i=1}^n q_i^2\biggl\{\sum_{j \neq i} \Delta_{\lambda_j}^{\lambda_i} \partial_{\lambda_j} B(\lambda_i,\lambda_j)  \biggr\}
  =
  \sum_{j=1}^n 
  \sum_{i=1}^n 
  q_i^2\biggl\{\Delta_{\lambda_j}^{\lambda_i} \partial_{\lambda_j} B(\lambda_i,\lambda_j)  \biggr\}
  =0,
\]
where we have extended the formula to $i = j$ by continuity, and the second equality follows as the inner sum is identically $0,$ again by degree considerations.  This formula allows us to write
\[
\sum_{i=1}^n \biggl\{\sum_{j \neq i}\frac{1}{\lambda_i - \lambda_j} \biggr\} \cdot q_i^2 \cdot G(\lambda_i) 
 =\frac{-1}{2}
\sum_{i=1}^n \biggl\{\sum_{j \neq i}\frac{ 
  \partial_\lambda B(\lambda_i,\lambda) \vert_{\lambda=\lambda_i}
  +\partial_\lambda B(\lambda_i,\lambda) \vert_{\lambda=\lambda_j}
}{\lambda_i - \lambda_j} \biggr\} \cdot q_i^2 
\]
So, adding and subtracting $\partial_\lambda B(\lambda_j,\lambda) \vert_{\lambda=\lambda_j},$ we can write
\begin{equation}
 \sum_{i=1}^n \biggl\{\sum_{j \neq i}\frac{1}{\lambda_i - \lambda_j} \biggr\} \cdot q_i^2 \cdot G(\lambda_i) 
 =\frac{-1}{2}
 \sum_{i,j:i\neq j}^n
 \Delta_{\lambda_i}^{\lambda_j} \partial_{\lambda_j} B(\lambda_i,\lambda_j) \cdot q_i^2
 +\frac{1}{2}
 \sum_{i,j:i\neq j}^n  \frac{G(\lambda_j)+G(\lambda_i)}{\lambda_i-\lambda_j} \cdot q_i^2
 \label{eq:S2}
\end{equation}
Rewriting the model, we have
\begin{equation}
  \begin{aligned}
  dA_t - dA_t' &= -2A_t\,dt + \frac{n}{2}
 \sum_{i,j=1}^n
 \Delta_{\lambda_i}^{\lambda_j} \partial_{\lambda_j} B(\lambda_i,\lambda_j) \cdot q_i^2 \cdot q_j^2 \\
 &- \frac{n}{2}
 \sum_{i,j=1}^n
 \Delta_{\lambda_i}^{\lambda_j} \partial_{\lambda_j} B(\lambda_i,\lambda_j) \cdot q_i^2 \cdot \biggl\{ q_j^2 - \frac{1}{n} \biggr\} \\
 &+\frac{1}{2}
 \sum_{i,j:i\neq j}^n  \frac{G(\lambda_j)+G(\lambda_i)}{\lambda_i-\lambda_j} \cdot q_i^2 \\
 &-
 \frac{1}{2}
 \sum_{i=1}^n
 \partial_{\lambda_i\lambda_j} B(\lambda_i,\lambda_j) \vert_{\lambda_j=\lambda_i} \cdot q_i^2 \\
 \label{eq:S3}
 \end{aligned}
\end{equation}
The first line can be expressed in terms of the minor polynomials.  Specifically, we define 
\(
G^{(0)}(\lambda)
=
-
 \sum_{i=1}^n
 \Delta_{\lambda_i}^{\lambda} \partial_{\lambda} B(\lambda_i,\lambda) \cdot q_i^2,
\)
in terms of which
\begin{equation}
  G_{k,\ell}^{(0)}(\lambda)
  =-\begin{cases}
    -a_{\ell-1} p^{(0)}_{\ell-2}(\lambda)p_{\ell-1}'(\lambda) + a_\ell p^{(0)}_{\ell-1}(\lambda) p_{\ell}'(\lambda)&  \\
    -a_{\ell-1} p^{(0)}_{\ell-1}(\lambda)p_{\ell-2}'(\lambda) + a_\ell p^{(0)}_{\ell}(\lambda) p_{\ell-1}'(\lambda),  & \text{ if } k = \ell , \\
    -a_\ell (p^{(0)}_{\ell-1}(\lambda)p_{\ell-1}'(\lambda) - p^{(0)}_\ell(\lambda)p_\ell'(\lambda)), & \text{ if } k = \ell + 1 ,\\
    0, & \text{ otherwise.}
  \end{cases}
 \label{eq:BB}
\end{equation}
We will make a comparison between the minor polynomials and the original polynomials to expose the leading order behavior (c.f.\,\eqref{eq:MinorPerturb}).  The last line is an absolute constant multiple of $dR^3,$ which is negligible in the limit (c.f.\,Lemma~\ref{lem:R3}).  

As for the middle two lines, they can be expressed in terms of the fluctuations of the spectral weights:  
\begin{lemma}
  Let $\gamma$ be a smooth contour enclosing all the eigenvalues once in the positive orientation.
  Then
   \[
    \begin{aligned}
 -\sum_{i,j=1}^n
 \Delta_{\lambda_i}^{\lambda_j} \partial_{\lambda_j} B(\lambda_i,\lambda_j) \cdot q_i^2 \cdot \biggl\{ q_j^2 - \frac{1}{n} \biggr\} 
 &= \frac{1}{2\pi i}\oint (s^A_t(z) - s_t(z))G^{(0)}(\sqrt{n} z)\,dz \\ 
    \end{aligned}
  \]
  and
  \[
    \begin{aligned}
    \sum_{i,j:i\neq j}^n  \frac{G(\lambda_j)+G(\lambda_i)}{\lambda_i-\lambda_j} \cdot 
    q_i^2
    &=\frac{-\sqrt{n}}{(2\pi i)^2}\oint\oint  (s^A_t(z) - s_t(z))s_t(y) \frac{G(\sqrt{n} z)-G(\sqrt{n} y)}{z-y}\,dzdy \\
    &+\frac{2\sqrt{n}}{2\pi i}\oint (s^A_t(z) - s_t(z))s_t(z) G(\sqrt{n} z)\,dz,
    \end{aligned}
  \]
  where the contour integrals are over $\gamma.$
  \label{lem:second}
\end{lemma}
\begin{proof}
  By antisymmetry, we have that
  \[
    \sum_{i,j:i\neq j}^n  \frac{G(\lambda_j)+G(\lambda_i)}{\lambda_i-\lambda_j} \cdot q_i^2
    =
    \sum_{i,j:i\neq j}^n  \frac{G(\lambda_j)+G(\lambda_i)}{\lambda_i-\lambda_j} \cdot 
    \biggl(q_i^2-\frac{1}{n}\biggr).
  \]
  This we now break into two parts
  \[
    \sum_{i,j:i\neq j}^n  \frac{G(\lambda_j)+G(\lambda_i)}{\lambda_i-\lambda_j} \cdot 
    \biggl(q_i^2-\frac{1}{n}\biggr)
    =
    -\sum_{i,j:i\neq j}^n  \frac{G(\lambda_i)-G(\lambda_j)}{\lambda_i-\lambda_j} \cdot 
    \biggl(q_i^2-\frac{1}{n}\biggr)
    +\sum_{i,j:i\neq j}^n  \frac{2G(\lambda_i)}{\lambda_i-\lambda_j} \cdot 
    \biggl(q_i^2-\frac{1}{n}\biggr)
  \]
  For entire functions $f,$ by a residue computation
  \begin{equation}
    \label{eq:residues}
    \begin{aligned}
      \frac{1}{2\pi i} \oint (s^A_t(z) - s_t(z))f(z)\,dz
      =
      &\sum_{i=1}^n \biggl\{q_i^2-\frac{1}{n}\biggr\}f(\lambda_{i,t}/\sqrt{n}), \text{ and } \\
      \frac{1}{2\pi i} \oint (s^A_t(z) - s_t(z)) s_t(z) f(z)\,dz
      =
      &\frac{1}{\sqrt{n}}
      \sum_{i=1}^n \biggl\{q_i^2-\frac{1}{n}\biggr\}\cdot \biggl\{\sum_{j \neq i}\frac{f(\lambda_{i,t}/\sqrt{n})}{\lambda_{i,t}-\lambda_{j,t}}\biggr\}
      \\
      +&\frac{1}{2n}
      \sum_{i=1}^n \biggl\{q_i^2-\frac{1}{n}\biggr\}f'(\lambda_{i,t}/\sqrt{n}),
    \end{aligned}
  \end{equation}
  for a contour enclosing all the eigenvalues once in the positive orientation. Hence, appropriately accounting for $G'$ terms, we get the desired formula.
\end{proof}

\subsection{Cancellation at leading order}

We consider the leading order behavior of the sum
\[
 \sum_{i,j=1}^n
 \Delta_{\lambda_i}^{\lambda_j} \partial_{\lambda_j} B(\lambda_i,\lambda_j) \cdot q_i^2\cdot q_j^2.
\]
On account of \eqref{eq:BB}, we are led to consider the following asymptotics. 
\begin{lemma}

  \begin{align*}
    &\sum_{j=1}^n p_k^{(0)}(\lambda_j) p_{k}'(\lambda_j)\cdot q_j^2 = 
    \frac{k}{a_1 a_k} 
    + n^{-3/2}\sum_{j=1}^{k-1} 16 (a_j-\tfrac{\sqrt{n}}{2}) (k-2j) 
    + n^{-3/2} 8k(a_1-a_k)
    + o(n^{-3/2}),
    \\
  &\sum_{j=1}^n p_k^{(0)}(\lambda_j) p_{k+1}'(\lambda_j)\cdot q_j^2 
  =  \frac{(b_1+\cdots + b_k-kb_{k+1}) }{a_1a_k a_{k+1}}
    + n^{-3/2}\sum_{j=1}^{k} 8b_j(k+1-2j)
  +o(n^{-3/2}),\\
    &\sum_{j=1}^n p_k^{(0)}(\lambda_j) p_{k-1}'(\lambda_j)\cdot q_j^2
    = n^{-3/2}\sum_{j=1}^{k} 8b_j(k+1-2j)
    + o(n^{-3/2}),
  \end{align*}
  with the errors uniform in $0 \leq t \leq T \wedge \sigma_K \wedge \tau_R.$ 
  \label{lem:mpeval}
\end{lemma}
\begin{proof}
  We begin with the first asymptotic.
  Equation \eqref{eq:MinorPerturb} states that to leading order,
  \[
    p_{k}^{(0)}(x) = \frac{p_{k-1}{(x)}}{a_1}.
  \]
  Hence, for the first identity, we have
  \begin{align*}
    \sum_{j=1}^n 
    \frac{p_{k-1}{(\lambda_j)}}{a_1}
    p_k'(\lambda_j)
    \cdot q_j^2
    &=
    \frac{k }{a_1a_k}. 
  \end{align*}

  On the other hand, we also have the subleading behavior of $p_{k}^{(0)}$ given by
  \begin{equation*}
  \begin{aligned}
  p^{(0)}_{k}(x) - \frac{1}{a_1}p_{k-1}(x)
  &=
  \frac{a_{k-1}-a_k}{a_{k}a_1}p_{k-1}(x) \\
  &+
  \frac{a_{k-1}}{a_{k}}\sum_{j=1}^{k-1} (b_{j} - b_{j+1})p^{(j-1)}_{k-1}(x)p_{j}^{(0)}(x) 
  \\
  &+
  \frac{a_{k-1}}{a_{k}}\sum_{j=1}^{k-2} (a_{j} - a_{j+1})
  (
  p^{(j)}_{k-1}(x)p_{j}^{(0)}(x) 
  +p^{(j-1)}_{k-1}(x)p_{j+1}^{(0)}(x) 
  ). 
\end{aligned}
\end{equation*}
Hence, on scaling up by $n,$ we get locally uniformly
  \begin{equation}
  \label{eq:perturbationlimit}
  \begin{aligned}
    \frac{n}{4}(p^{(0)}_{k}(x) - \frac{1}{a_1}p_{k-1}(x))
  &\to
  (a_{k-1}-a_k)U_{k-1}(x) \\
  &+
  \sum_{j=1}^{k-1} (b_{j} - b_{j+1})U_{k-j-1}(x)U_{j-1}(x)
  \\
  &+
  \sum_{j=1}^{k-2} (a_{j} - a_{j+1})
  (
  U_{k-j-2}(x)U_{j-1}(x)
  +U_{k-j-1}(x)U_{j}(x)
  ). 
\end{aligned}
\end{equation}
We then need to evaluate the integral of this asymptotic against $U_k'(x)\mathfrak{s}(x)\,dx,$ for which purpose, we will need the following Chebyshev polynomial identities
\begin{align}
  U_k'(x) &= \sum_{j=0}^{k/2} 2(k-2j) U_{k-2j-1}(x) \label{eq:dU} \\
  U_k(x)U_\ell(x) &= \sum_{j=0}^{\ell} U_{k-\ell+2j}(x), \text{ if } k \geq \ell. \label{eq:UU}
\end{align}
By parity considerations, some terms cancel.  For any $1 \leq j \leq k-1,$ $U_{k-j-1}(x)U_{j-1}(x)$ is a sum over $U_{p}(x)$ where $p$ has the same parity as $k.$  In particular, it follows that
\[
  \int_{-1}^{1} U_k'(x)
  \left\{
  \sum_{j=1}^{k-1} (b_{j} - b_{j+1})U_{k-j-1}(x)U_{j-1}(x)
\right\}
  \mathfrak{s}(x)\,dx
  =0.
\]


On the other hand, for $1 \leq j \leq k-1$
\begin{align*}
  &\int_{-1}^{1} U_k'(x)
  (
  U_{k-j-2}(x)U_{j-1}(x)
  +U_{k-j-1}(x)U_{j}(x)
  )
  \mathfrak{s}(x)\,dx \\
  &= 
  \sum_{\ell=1}^{ j \wedge (k-j-1) } 2(k-2\ell)
  +\sum_{\ell=0}^{ j \wedge (k-j-1) } 2(k-2\ell).
\end{align*}
Define the array, for $1 \leq j < k-1$
  \[
    H_{k,j} = \sum_{\ell=1}^{j \wedge (k-j-1)} 2(k-2\ell),
  \]
  and $H_{k,j} = 0$ for $j \geq k-1.$
  Then we can express
\[
  n^{3/2}\sum_{j=1}^n p_k'(\lambda_j)(p^{(0)}_{k}(\lambda_j) - \frac{1}{a_1}p_{k-1}(\lambda_j))q_j^2 \to 
  \sum_{j=1}^{k-1} 4(a_j - a_{j+1})(2H_{k,j}+2k).
\]

For the second term we begin with the observation that for any $k \in \N$
\begin{equation}
\label{eq:pkcoeffiecients}
p_k(x) = \frac{1}{a_1\cdots a_k} x^k - \frac{b_1+\cdots +b_k}{a_1\cdots a_k}x^{k-1} + ...
\end{equation}
Therefore,
\[
  p_{k+1}'(x) - \frac{(k+1)p_k(x)}{a_{k+1}}
  = 
  -k\frac{b_1+\cdots +b_{k+1}}{a_1\cdots a_{k+1}}x^{k-1}
  +(k+1)\frac{b_1+\cdots +b_{k}}{a_1\cdots a_{k+1}}x^{k-1}
  +O(x^{k-2})
\]
Hence,
\begin{align*}
  n^{3/2}
\sum_{j=1}^n 
    \frac{p_{k-1}{(\lambda_j)}}{a_1}
    p_{k+1}'(\lambda_j)\cdot q_j^2
    =
    \frac{n^{3/2}(b_1+\cdots + b_k-kb_{k+1}) }{a_1a_k a_{k+1}}.
  \end{align*}
We again use equation \eqref{eq:perturbationlimit} but observe that the parity is the opposite of the previous case which leaves us with the terms of the form
  \[
    \begin{aligned}
      \frac{n^{3/2}}{4}
    &\sum_{j=1}^n 
    \biggl(p^{(0)}_{k}(\lambda_i) - \frac{p_{k-1}(\lambda_i)}{a_1}\biggr)
    p_{k+1}'(\lambda_j)\cdot q_j^2 \\
    &\to
      \int_{-1}^{1} U_{k+1}'(x)
  \left\{
  \sum_{j=1}^{k-1} (b_{j} - b_{j+1})U_{k-j-1}(x)U_{j-1}(x)
\right\}
  \mathfrak{s}(x)\,dx \\
  &=\sum_{j=1}^{k-1}(b_j-b_{j+1})\sum_{m=1}^{(k-j)\wedge(j)}2(k+1-2m).
    \end{aligned}
  \]

For the final term, we have by degree considerations that
\[
  n^{3/2}
\sum_{j=1}^n 
    \frac{p_{k-1}{(\lambda_j)}}{a_1}
    p_{k-1}'(\lambda_j)\cdot q_j^2
    =
    0.
\]
And we have that
  \[
    \begin{aligned}
      \frac{n^{3/2}}{4}
    &\sum_{j=1}^n 
    \biggl(p^{(0)}_{k}(\lambda_i) - \frac{p_{k-1}(\lambda_i)}{a_1}\biggr)
    p_{k-1}'(\lambda_j)\cdot q_j^2 \\
    &\to
      \int_{-1}^{1} U_{k-1}'(x)
  \left\{
  \sum_{j=1}^{k-1} (b_{j} - b_{j+1})U_{k-j-1}(x)U_{j-1}(x)
\right\}
  \mathfrak{s}(x)\,dx \\
  &=\sum_{j=1}^{k-1}(b_j-b_{j+1})\sum_{m=1}^{(k-j)\wedge(j)}2(k+1-2m),
    \end{aligned}
  \]
  where we observe the limit agrees with that of the previous case.
  
  Lastly we use the identity
 \begin{equation}
 \label{eq:Hdiff}
 H_{k,j}-H_{k,j-1} = 2(k-2j)
 \end{equation}
 for $j<k$ to simplify the solution.
\end{proof}

\begin{corollary}
  Define the operator $\mathcal{F}$ on Jacobi matrices by
  \[
    \mathcal{F}(A)_{k,\ell}
    =
    \begin{cases}
      -b_\ell(4\ell-2)+
       4\sum_{j=1}^{\ell-1} b_j & \text{ if } k = \ell , \\
      -2 (\ell+1)a_\ell- 2\ell a_{\ell-1} 
      + 4\sum_{j=1}^{\ell-2} a_j 
      & \text{ if } k = \ell + 1 ,\\
    0, & \text{ otherwise.}
    \end{cases}
  \]
  Uniformly in $1 \leq k,\ell \leq M$ and $0 \leq t \leq T \wedge \sigma_K \wedge \tau_R.$
  \begin{equation}
    -2A_t +\frac{n}{2}
    \sum_{i,j=1}^n
    \Delta_{\lambda_i}^{\lambda_j} \partial_{\lambda_j} B_{k,\ell}(\lambda_i,\lambda_j) \cdot q_i^2 \cdot q_j^2 
    -\mathcal{F}(A_t - D\sqrt{n})_{k,\ell} 
     \to 0
  \end{equation}
  \label{cor:lead}
\end{corollary}
\begin{proof}
  We arrive at these asymptotics substituting the asymptotics of Lemma~\ref{lem:mpeval} into \eqref{eq:BB}. 
\end{proof}

\subsection{Convergence of the model}

Recalling \eqref{eq:S3} for convenience
\begin{equation*}
  \begin{aligned}
  dA_t - dA_t' &= -2A_t\,dt + \frac{n}{2}
 \sum_{i,j=1}^n
 \Delta_{\lambda_i}^{\lambda_j} \partial_{\lambda_j} B(\lambda_i,\lambda_j) \cdot q_i^2 \cdot q_j^2 \\
 &- \frac{n}{2}
 \sum_{i,j=1}^n
 \Delta_{\lambda_i}^{\lambda_j} \partial_{\lambda_j} B(\lambda_i,\lambda_j) \cdot q_i^2 \cdot \biggl\{ q_j^2 - \frac{1}{n} \biggr\} \\
 &+\frac{1}{2}
 \sum_{i,j:i\neq j}^n  \frac{G(\lambda_j)+G(\lambda_i)}{\lambda_i-\lambda_j} \cdot q_i^2 \\
 &-
 \frac{1}{2}
 \sum_{i=1}^n
 \partial_{\lambda_i\lambda_j} B(\lambda_i,\lambda_j) \vert_{\lambda_j=\lambda_i} \cdot q_i^2 \\
 \end{aligned}
\end{equation*}
we have shown $dA'_t$ is negligible in the limit (Proposition~\ref{prop:dAt}).  The final line of this equaltity is negligible by the same argument, after identifying it as a multiple of $dR^3_t$ for the entries we consider.  In Corollary \ref{cor:lead} we have identified the asymptotics of the first line, showing it is essentially a linear function of the entries of $A_t-D\sqrt{n}.$  Finally, in Lemma~\ref{lem:second} we have represented the middle two lines as contour integrals against $s_t^A(z) -s_t(z)$ and $s_t(z).$  By the convergenece of these Stieltjes transforms (Proposition~\ref{prop:stieltjes}) and the various polynomials (Lemma~\ref{lem:cheb}), the contour integrals in Lemma~\ref{lem:second} converge to time--independent contour integrals.  All that remains to complete the proof is identify these limiting contour integrals.
In summary, we have that
\begin{equation}
  \begin{aligned}
    dA_t &= \mathcal{F}(A_t - D\sqrt{n}) \\
  &+\frac{n}{4\pi i}\oint (s^A_t(z) - s_t(z))G^{(0)}(\sqrt{n} z)\,dz \\ 
  &+\frac{-\sqrt{n}}{2(2\pi i)^2}\oint\oint  (s^A_t(z) - s_t(z))s_t(y) \frac{G(\sqrt{n} z)-G(\sqrt{n} y)}{z-y}\,dzdy \\
  &+\frac{\sqrt{n}}{2\pi i}\oint (s^A_t(z) - s_t(z))s_t(z) G(\sqrt{n} z)\,dz + o(1)
  \label{eq:S4}
\end{aligned}
\end{equation}
with the error vanishing uniformly over bounded order principal minors and uniformly in $0 \leq t \leq T \wedge \tau_R \wedge \sigma_K.$ 

Define the semi--infinite Jacobi matrix
\begin{equation}
    \UM(x,y)_{k,\ell}
    =
    \frac{1}{2}
    \begin{cases}
      - U_{\ell-2}(x)U_{\ell-1}(y) + U_{\ell-1}(x) U_{\ell}(y)&  \\
      -U_{\ell-1}(x)U_{\ell-2}(y) +U_{\ell}(x) U_{\ell-1}(y),  & \text{ if } k = \ell, \\
      -U_{\ell-1}(x)U_{\ell-1}(y) + U_\ell(x)U_\ell(y), & \text{ if } k = \ell + 1,\\
      0, & \text{ otherwise,}
    \end{cases}
    \label{eq:UM}
\end{equation}
All expressions involving $G$ and $B$ scale to expressions involving $\UM.$  For example, we have that uniformly in bounded order principal minors and uniformly in $0 \leq t \leq T \wedge \tau_R \wedge \sigma_K$ 
\[
  \frac{\sqrt{n}}{2\pi i}\oint (s^A_t(z) - s_t(z))s_t(z) G(\sqrt{n} z)\,dz
  \to
  \frac{-1}{4\pi i}\oint \SQ(z)\SCS(z) \partial_z \UM(z,z)\,dz.
\]
A similar limit can be given for the second contour integral of \eqref{eq:S4}.  
Specifically,
\[
  \begin{aligned}
  \frac{-\sqrt{n}}{2(2\pi i)^2}\oint\oint  (s^A_t(z) - s_t(z))s_t(y) 
  &\frac{G(\sqrt{n} z)-G(\sqrt{n} y)}{z-y}\,dzdy \\
  &\to
  \frac{1}{2(2\pi i)^2}\oint\oint 
  \SQ(z)\SCS(y)
  (\Delta_z^y + \Delta_y^z)\partial_z\UM(y,z)
  \,dydz.\\
\end{aligned}
\]
By degree considerations,
\[
  \frac{1}{2\pi i}\oint
  \SCS(y)
  \Delta_z^y\partial_z\UM(y,z)\,dy
  =\int_{-1}^{1}\Delta_z^y\partial_z\UM(y,z)\mathfrak{s}(y)\,dy=0.
\]
As for the first contour integral of \eqref{eq:S4}, it converges to an expression involving the minor polynomials of the Chebyshev polynomials of the second kind, which we can express using contour integrals as
\[
  \frac{n}{4\pi i}\oint (s^A_t(z) - s_t(z))G^{(0)}(\sqrt{n} z)\,dz 
  \to
  \frac{-1}{2(2\pi i)^2}\oint\oint 
  \SQ(z)\SCS(y)
  \Delta_y^z \partial_z\UM(y,z)
  \,dydz.
\]
In summary we have $A_t - \sqrt{n}D$ converges to the solution of 
\begin{equation}
  \begin{aligned}
  dA_t &= \mathcal{F}(A_t) 
  -\frac{1}{4\pi i}\oint \SQ(z)\SCS(z) \partial_z \UM(z,z)\,dz.
\end{aligned}
\end{equation}

{\small

\bibliographystyle{abbrv}
\bibliography{rmt}

}

\end{document}